\documentclass[10pt, reqno]{amsart}

\usepackage{graphicx}
\usepackage{amsmath,amssymb,mathtools,amsopn,amsfonts,amsthm}
\usepackage{bbm}
\usepackage{subfigure}
\usepackage{eepic}
\usepackage{tikz}
\usetikzlibrary{calc,angles,positioning,intersections,quotes}
\usepackage{a4wide}
\usepackage[utf8]{inputenc}
\usepackage{epsfig}
\usepackage{latexsym}
\usepackage[shortlabels]{enumitem}
\usepackage[UKenglish]{babel}
\usepackage{verbatim}
\usepackage{xcolor}

\usepackage[initials]{amsrefs}
\usepackage[bookmarks]{hyperref}

\numberwithin{equation}{section}
\numberwithin{figure}{section}
\theoremstyle{plain}
\newtheorem{thm}{\protect\theoremname}[section]
\theoremstyle{plain}
\newtheorem{lem}[thm]{\protect\lemmaname}
\theoremstyle{definition}
\newtheorem{defn}[thm]{\protect\definitionname}
\theoremstyle{plain}
\newtheorem{cor}[thm]{\protect\corollaryname}
\theoremstyle{plain}
\newtheorem{prop}[thm]{\protect\propositionname}
\theoremstyle{plain}
\newtheorem*{lem*}{\protect\lemmaname}
\theoremstyle{plain}
\newtheorem*{thm*}{\protect\theoremname}
\newtheorem{rem}[thm]{\protect\remarkname}
\theoremstyle{plain}

\usepackage{babel}
\providecommand{\corollaryname}{Corollary}
\providecommand{\definitionname}{Definition}
\providecommand{\lemmaname}{Lemma}
\providecommand{\propositionname}{Proposition}
\providecommand{\theoremname}{Theorem}
\providecommand{\remarkname}{Remark}
\newcommand{\A}{\mathbf{A}}
\newcommand{\Se}{\mathbf{S}}
\newcommand{\N}{\mathbb{N}}

\newcommand{\R}{\mathbb{R}}
\renewcommand{\P}{\mathbb{P}}
\newcommand{\E}{\mathbb{E}}
\newcommand{\D}{\mathcal{D}}

\newcommand{\Z}{\mathbb{Z}}

\newcommand{\SL}{\mathrm{SL}(3,\mathbb{R})}
\newcommand{\GL}{\mathrm{GL}(3,\mathbb{R})}

\providecommand{\norm}[1]{\lVert#1\rVert}

\newcommand{\I}{\mathcal{I}}
\newcommand{\J}{\mathcal{J}}
\renewcommand{\i}{\mathbf{i}}
\renewcommand{\j}{\mathbf{j}}
\renewcommand{\k}{\mathbf{k}}

\newcommand{\hd}{\dim_\textup{H}}
\newcommand{\bd}{\dim_\textup{B}}

\newcommand{\ld}{\dim_\textup{LY}}

\title{Hausdorff dimension of the Rauzy gasket}

\author{Natalia Jurga} \address{Mathematical Institute, University of St Andrews, Scotland, KY16 9SS}
\email{naj1@st-andrews.ac.uk}

\begin{document}

 \maketitle

\begin{abstract}
The Rauzy gasket is the attractor of a parabolic, nonconformal iterated function system on the projective plane which describes an exceptional parameter set in various important topological and dynamical problems. Since 2009 there have been several attempts to calculate the Hausdorff dimension of the Rauzy gasket, which is a challenging problem due to the combination of the parabolicity and nonconformality in the geometry. In this paper we settle this question by proving that the Hausdorff dimension of the Rauzy gasket equals the (projective) affinity dimension. The key technical result underpinning this is a partial generalisation of work of Hochman and Solomyak to the $\SL$ setting, where we establish the exact dimension of stationary (Furstenberg) measures supported on the Rauzy gasket. The dimension results for both stationary measures and attractors are established in broader generality and extend recent work on projective iterated function systems to higher dimensions.
\end{abstract}

\section{Introduction}

Let $\A_R$ denote the set of matrices $\A_R=\{A_1,A_2,A_3\}$ where
\begin{equation}
\label{rauzysystem}
A_1=\begin{pmatrix} 1&1&1 \\0&1&0\\ 0&0&1 \end{pmatrix} ,\;\; A_2=\begin{pmatrix} 1&0&0 \\1&1&1\\ 0&0&1 \end{pmatrix}, \;\; A_3=\begin{pmatrix} 1&0&0 \\0&1&0\\ 1&1&1 \end{pmatrix}.
\end{equation}
The action of the matrices $A_i$  on the standard two-dimensional simplex $S=\{(x,y,z) : x,y,z \geq 0, x+y+z=1\}$ induces an iterated function system (IFS) $\{f_{A_i}: S \to S\}_{i=1}^3$ where 
\begin{align*}
f_{A_1}(x,y,z)&=\left(\frac{1}{2-x}, \frac{y}{2-x}, \frac{z}{2-x}\right)\\
f_{A_2}(x,y,z)&=\left(\frac{x}{2-2}, \frac{1}{2-y}, \frac{z}{2-y}\right)\\
f_{A_3}(x,y,z)&=\left(\frac{x}{2-z}, \frac{y}{2-z}, \frac{1}{2-z}\right).
\end{align*}
 The Rauzy gasket is the attractor of this IFS, i.e. the unique non-empty compact set $R \subset S$ for which $R=\bigcup_{i=1}^3 f_{A_i}(R)$.

The Rauzy gasket has been discovered independently three times in various topological and dynamical contexts. It first appeared in the work of Arnoux and Rauzy \cite{arnoux1991representation} in the context of interval exchange transformations. It was later rediscovered by Levitt \cite{levitt1993dynamique} where it was associated with the simplest example of a pseudogroup of rotations, and finally was reintroduced yet again in the work of Dynnikov and De Leo \cite{deleo2009geometry}  in connection with Novikov’s problem of plane sections of triply periodic surfaces. 

\begin{figure}[h!]
\centering
\includegraphics[width=5cm]{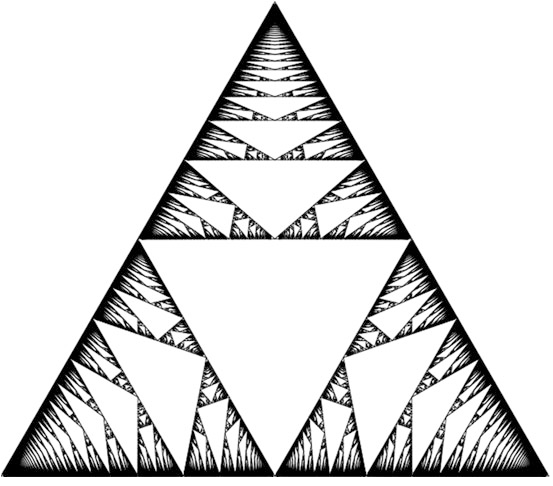}
\caption{The Rauzy gasket contained in the standard 2-simplex $S$.}
\end{figure}

In each of these contexts, the Rauzy gasket describes some parameter set of exceptional behaviour, leading to the natural question: what can one say about the `size' of the Rauzy gasket? There are of course many ways of quantifying size; the literature has focused on what is the Lebesgue measure $\mathcal{L}(R)$ and the more subtle question of what is the Hausdorff dimension $\hd R$.

In \cite{arnoux1991representation} it was conjectured that $\mathcal{L}(R)=0$ which was later proved in \cite{levitt1993dynamique} in a proof accredited to Yoccoz. \cite{deleo2009geometry} contained an alternative proof of this fact and it was conjectured that the Hausdorff dimension should be strictly less than 2. This was rigorously established by Avila, Hubert and Skripchenko \cite{avila2016diffusion}. Since then there have been a number of further papers estimating the Hausdorff dimension of the gasket \cite{fougeron2020dynamical, pollicott2021upper,gutierrez2020lower} which have yielded the bounds $1.19< \hd R < 1.74$. 

There are several reasons why the Hausdorff dimension of the Rauzy gasket has been so challenging to compute. While the dimension theory of conformal and uniformly contracting IFSs are very well understood, the IFS generating the Rauzy gasket is neither conformal (as can be seen by the fact that the copies of $R$ get increasingly distorted as one goes deeper into the construction) nor uniformly contracting (as can be seen by the parabolicity coming from the fixed points). While the dimension theory of \emph{conformal} parabolic IFSs has been widely developed by Mauldin, Urba\'nski and coauthors e.g. \cite{mauldin2000parabolic}, and \emph{affine} nonconformal IFSs are becoming increasingly better understood e.g. \cite{BHR,HR,rapaport2022SelfAffRd}, the understanding of nonconformal nonlinear IFSs is still fairly limited, while simultaneously nonconformal and parabolic IFSs have hardly been studied at all.
%, where the parabolic carpets in \cite{fraser2022parabolic} being the only example that the author is aware of.

While much of the literature on the Rauzy gasket has focused on obtaining numerical bounds using ad hoc methods, there has been no systematic treatment of the Rauzy gasket within its natural broader context: through the geometric measure theory of the action of matrix semigroups on projective space. Throughout this paper we will restrict our attention to strongly irreducible and proximal subsemigroups of $\mathrm{SL}(d,\R)$, and in most of the paper $d=3$ although we keep the discussion general for now. We say $\A$ is irreducible if there is no proper linear subspace which is preserved by each matrix in $\A$, strongly irreducible if there is no finite union of proper linear subspaces which is preserved by each matrix in $\A$ and proximal if the semigroup $\Se$ generated by $\A$ contains a proximal element, i.e. a matrix with a simple leading eigenvalue. Given a proximal and strongly irreducible semigroup $\Se$ generated by $\A \subset \mathrm{SL}(d,\R)$, its (projective) limit set $K_\A \subset \R\P^{d-1}$ can be defined as the closure of the set of attracting fixed points of proximal elements in $\Se$ see eg. \cite[\S 3]{BQ}. Throughout the paper we will refer to a set in $\R\P^{d-1}$ which can be realised as the projective limit set of some semigroup $\Se$ as a self-projective set. While these are generally subsets of $(d-1)$-dimensional real projective space $\R\P^{d-1}$, when $\A$ are positive matrices it is sometimes simpler to study the limit set by identifying directions in $\R^d$ with a suitable linear subspace. We have already seen that one way to do this is to study the limit set within the simplex $S$, whereas in this paper it will be more convenient to consider the induced projective maps on the plane $\{(x,y,1): x,y>0\}$ and study the induced limit set there, see \S \ref{lft}. Under positivity of $\A$ the limit set is actually the attractor of an IFS, although this is not true of self-projective sets in general. 

Closely related to the projective limit sets are the family of stationary measures, often called Furstenberg measures \cite{BQ,BL}. Given a probability measure $\eta$ fully supported on $\A=\{A_1, \ldots,A_N\}$ there exists a unique stationary measure $\mu$ which is supported on the limit set, which describes the distribution of the random vector $A_{i_n} \cdots A_{i_1}v \in \R\P^{d-1}$ where each $A_{i_j}$ is chosen i.i.d with respect to $\eta$, see \S \ref{Slyap}. 

The dimension theory of both self-projective sets and their associated stationary measures have important applications. For example, the Hausdorff dimension of the limit set
is a key determiner in the asymptotics of counting problems for semigroups within the field of geometric group
theory \cite{magee2017uniform}, whereas understanding the “smoothness properties” of stationary measures is key to determining
the regularity properties of the integrated density of states for the Anderson-Bernoulli model of random
Schrödinger operators \cite{bourgain2014application} as well as being a central ingredient in the dimension theory of self-affine sets and
measures \cite{BHR,HR,rapaport2022SelfAffRd}.

By viewing the Rauzy gasket through the lens of the theory of self-projective sets and stationary measures, a wide range of both classical and modern sophisticated tools become available, most notably from the theory of self-affine sets, such as subadditive thermodynamic formalism, the recent pressure approximation techniques in \cite{morris2023variational} and the entropy methods of \cite{BHR,HR,rapaport2022SelfAffRd}. Using these techniques we resolve the question of the Hausdorff dimension of the Rauzy gasket. 

For $A \in \SL$ let $\alpha_1(A) \geq \alpha_2(A) \geq \alpha_3(A)$ denote the singular values of $A$. For $s \geq 0$ we define the singular value function $\phi^s: \SL \to \R_{>0}$ 
\begin{equation}
\phi^s(A)=
\begin{cases}
\left(\frac{\alpha_2(A)}{\alpha_1(A)}\right)^s &\mbox{if } 0 \leq s \leq 1 \\
\frac{\alpha_2(A)}{\alpha_1(A)}\left(\frac{\alpha_3(A)}{\alpha_1(A)}\right)^{s-1} & \mbox{if } 1 \leq s \leq 2\\
\left(\frac{\alpha_2(A) \alpha_3(A)}{\alpha_1(A)^2}\right)^s &\mbox{if } s \geq 2 \end{cases} 
\label{svf}
\end{equation}
and for a finite or countable set $\A \subset \SL$ we define the zeta function $\zeta_\A: [0,\infty) \to [0,\infty]$ by
\begin{equation} \label{zeta}
\zeta_\A(s)= \sum_{n=1}^\infty \sum_{A \in \A^n} \phi^s(A).
\end{equation}
Finally we define the (projective) \emph{affinity dimension} $s_\A$ to be the critical exponent of this series
$$s_\A=\inf\{s\geq 0: \zeta_\A(s)<\infty\}.$$
We delay giving some heuristic to these objects till the next section. The following is our first main result.

\begin{thm} \label{rauzythm}
\begin{equation}\hd R=s_{\A_R}= \inf\left\{s \geq 0: \sum_{n=1}^\infty \sum_{A \in \A_R^n} \frac{\alpha_2(A)}{\alpha_1(A)}\left(\frac{\alpha_3(A)}{\alpha_1(A)}\right)^{s-1} < \infty \right\}. \label{rauzydim} \end{equation}
\end{thm}

\eqref{rauzydim} resembles the qualitative type of dimension formulae that appear in the theory of limit sets of Fuchsian and Kleinian groups and the dimension theory of expanding repellers.  We remark that it is possible to characterise \eqref{rauzydim} in terms of the spectral radius of a ``transfer operator'' and appeal to the broad range of modern tools available to rigorously and accurately approximate $s_{\A_R}$ numerically via this spectral radius, e.g \cite{pollicott2022hausdorff, morris2022fast}. We do not pursue this within the current paper.
% but outline a possible approach in the appendix.

The key technical result underpinning Theorem \ref{rauzythm} is a result which computes the dimension of a certain class of stationary measures. In the case that $\mu$ is a stationary measure associated to a subsemigroup of $\mathrm{SL}(2,\R)$, exact dimensionality of $\mu$ and the value of the (exact) dimension of $\mu$ is known in many cases by the work of Hochman and Solomyak \cite{HS}. Under the assumption that $\A$ generates a Diophantine semigroup (Definition \ref{ESC}), Hochman and Solomyak computed a formula for the dimension in terms of the Shannon entropy of $\eta$ and the top Lyapunov exponent. This formula is known as the Lyapunov dimension $\ld \mu$ and has a natural extension to higher dimensions, see \cite[Section 1.3]{rapaport2020exact} and \eqref{lyapdim} in the special case that $d=3$.  Rapaport recently established the exact-dimensionality of the stationary measure associated to a subsemigroup of $\mathrm{SL}(d,\R)$ when $d \geq 3$ and showed that its exact dimension is always bounded above by its Lyapunov dimension. However no higher-dimensional analogue of \cite{HS} (i.e. no matching lower bound) has been established. Here we obtain such an analogue in the special case that $\A \subset \SL_{>0}$ (the subsemigroup of matrices in $\mathrm{SL}(3,\mathbb{R})$ with all positive entries) and where the Diophantine property is replaced by the stronger separation condition known as the strong open set condition (SOSC), see Definition \ref{sosc}.

\begin{thm} \label{msrthm}
Suppose $\eta$ is supported on a finite set $\A \subset \SL_{>0}$ which generates a semigroup $\Se$ which is Zariski dense in $\SL$ and satisfies the SOSC. Then $\hd \mu=\ld \mu$.
\end{thm}

We note that it should be possible to relax the SOSC to the Diophantine property and to eliminate the assumption of positivity, however Zariski density cannot be relaxed to e.g. the assumption of proximality and strong irreducibility. To see this, consider the  non-Zariski dense group $SO(2,1)$ which preserves the bilinear form $B(\mathbf{x},\mathbf{y})=x_1y_1+x_2y_2 -x_3y_3$, and thus preserves the curve $C=\{\mathbf{x} \in \R\P^2: x_1^2+x_2^2=x_3^2\}$. One can choose $\mathrm{supp}\,\eta=\A=\{A_1, \ldots,A_N\}$ such that: (i) the entries of each $A_i$ are algebraic (which implies that $\A$ is Diophantine \cite{HS}), (ii) $\A$ contains a proximal element, (iii) $\A$ is strongly irreducible, (iv) sufficiently many matrices $A_i$ are chosen so that $s_\A>1$. Then $\A$ generates a semigroup which is strongly irreducible, proximal and  Diophantine but not Zariski dense in $\SL$, and since $\mathrm{supp} \mu \subset C$ we have $\hd \mu \leq 1<s_\A$.

As in \cite{HS}, the proof of Theorem \ref{msrthm} employs entropy methods which were initially developed to prove analogous results for self-similar measures and self-affine measures e.g. \cite{Ho,BHR}. While \cite{HS} required adapting Hochman's work on self-similar measures \cite{Ho} to the projective line, the nonconformality of the projective action for $d \geq 3$ means that to tackle measures for subsemigroups of $\SL$, it is the much more involved results for self-affine measures \cite{BHR,HR,rapaport2022SelfAffRd} which must be adapted to the projective plane. 

Finally, we remark that Theorem \ref{msrthm} can be used to prove a general result about the Hausdorff dimension of self-projective sets in the projective plane. The dimension theory of self-projective sets in the projective line is relatively well understood by the recent work \cite{solomyak2021diophantine,christodoulou2020hausdorff}, where the state of the art result determines that if $\A$ generates a Diophantine semigroup $\Se$, $\mathrm{Id} \notin \overline{\Se}$ and $K_\A$ is not a singleton then $\hd K_\A=\min\{1,s_\A\}$. Relaxing the Diophantine property to freeness of $\Se$ in this statement is equivalent to proving that if $\A \subset \mathrm{SL}(2,\R)$ freely generates a semigroup then $\hd K_\A=\min\{1,s_\A\}$, which can be considered a projective extension of the well-known `exact overlaps conjecture' \cite{simon1996overlapping} for self-similar sets.
%, a problem which remains very difficult even in the simplest case that the Mobius representations of the matrices are similarity mappins. 
The nonconformality of self-projective sets in higher dimensional projective spaces causes their dimension theory to be even more subtle, since mechanisms other than non-freeness can cause the dimension to drop, such as failure of irreducibility. Here we obtain the following corollary to Theorem \ref{msrthm} which is the first dimension result for self-projective sets in higher dimensional projective spaces.

\begin{thm} \label{setthm}
Suppose a finite set $\A \subset \SL_{>0}$ generates a semigroup $\Se$ which is Zariski dense in $\SL$ and satisfies the SOSC. Then $\hd K_\A=\min\{s_\A,2\}$.
\end{thm}

Finally we note that Theorem \ref{rauzythm} does not follow directly from Theorem \ref{setthm} since $\A_R$ cannot be simultaneously conjugated to a set of positive matrices. However, we will deduce Theorem \ref{rauzythm} from Theorem \ref{setthm} by showing that $\hd R$ can be approximated from within by subsystems of $\A_R$ which are positive in some basis.

We now outline the remainder of the paper. \S \ref{prel} contains some preliminaries, including the definition of the induced projective maps we will be working with in \S \ref{lft} and a reduction of Theorem \ref{msrthm} to Theorem \ref{thm:ESC --> Delta =00003D correct val}, which makes it sufficient to compute only the dimension of certain projections of the stationary measure. \S \ref{entropylb}-\S \ref{finalproof} contain the proof of Theorem  \ref{thm:ESC --> Delta =00003D correct val}. The scheme of proof follows broadly similar lines to \cite{BHR,HR,rapaport2022SelfAffRd}, and where possible we will avoid duplicating  technical arguments from these papers by simply referring to precise places where identical arguments have appeared. We will give a more detailed outline of the proof of Theorem  \ref{thm:ESC --> Delta =00003D correct val} at the beginning of \S \ref{entropylb}.  In \S \ref{finalsection} we complete the proofs of Theorem \ref{setthm} and \ref{rauzythm}. 

Finally, while preparing this manuscript the recent related work of Jiao, Li, Pan and Xu \cite{li2023dimension,jiao2023dimension} was brought to the authors' attention. Across these two papers the authors obtain an independent proof of Theorem \ref{rauzythm} via a generalised version of Theorem \ref{msrthm}, where positivity of the matrices is not required and the SOSC is relaxed to the Diophantine property. A broadly similar scheme of proof is employed in \cite{jiao2023dimension,li2023dimension} to prove Theorem \ref{msrthm}, however their work requires many extra technical arguments to deal with both a lack of global contraction properties and a lack of good separation properties. Thus we intend for the current paper to be a more straightforward and streamlined reference for readers interested in the fundamental essence of the arguments behind the proof of Theorem \ref{rauzythm}. \\

\noindent \textbf{Acknowledgement.} The author is deeply grateful to Ariel Rapaport for many inspiring discussions and guidance related to this project and for his proof of Theorem \ref{thm:ent inc under conv}. The research was supported by a Leverhulme Early Career Fellowship (ECF-2021-385).

\section{Preliminaries} \label{prel}

Let $m,n\ge2$ and write $\mathrm{Mat}_{m,n}(\mathbb{R})$ for the
vector space of all $m\times n$ real matrices. Throughout the paper for $A \in \mathrm{Mat}_{m,n}(\mathbb{R})$ we let $\norm{A}$ denote the operator norm of $A$.

We will write $X =\Theta_Z (Y)$ to mean that there exists a constant $c>0$ such that $c^{-1}Y \leq X \leq cY$, where $c$ may depend on $Z$ but is uniform in other parameters that appear in the statement.

Given an index set $\I$ we let $\I^n$ denote the set of words of length $n$, $\I^*$ denote the set of all finite words and $\I^\N$ denote the set of all infinite sequences, with digits in $\I$. We let $|\I|$ denote the size of the index set and for $\i=i_1 \ldots i_n \in \I^*$ we let $|\i|=n$ denote the length of the word $\i$. Given $\i=i_1 \ldots i_m \in \I^m$ with $m >n \geq 1$ we let $\i|n=i_1 \ldots i_n$ denote its truncation to its first $n$ digits. Given $\i \in\I^{n}$ we denote the corresponding cylinder set by 
\[
[\i]:=\{\omega\in\I^{\mathbb{N}}\::\:\omega|_{n}=\i\}.
\] $\I^\N$ is considered a topological space with the topology generated by cylinder sets. We will sometimes consider measures on $\I^\N$ which are invariant with respect to the left shift map $\sigma: \I^\N \to \I^\N$.

In \S \ref{alg}.1-\ref{alg}.3, we will keep the discussion general by letting $\A \subset \SL$ denote either a finite or countable set which generates a semigroup $\Se$, unless otherwise specified. Thereafter we will impose further assumptions on $\A$.

\subsection{Algebraic preliminaries}\label{alg}

\subsubsection{Exterior algebra}
Let $\wedge^2\R^3= \mathrm{span}\{e_{i_1} \wedge e_{i_2}: 1 \leq i_1 < i_2 \leq 3\}$. For $A \in \GL$ define the linear map $A^{\wedge 2}: \wedge^2\R^3 \to\wedge^2 \R^3$ by $(A^{\wedge 2})(e_{i_1} \wedge e_{i_2})=Ae_{i_1} \wedge Ae_{i_2}$ and extending by linearity. $A^{\wedge 2}$ can be represented by the $3 \times 3$ matrix whose entries are the $2 \times 2$ minors of $A$. We equip $\wedge^2 \R^3$ with the inner product $\langle v,w \rangle_2=*(v \wedge *w)$ where $*$ is the Hodge star operator $*:\wedge^2\R^3 \to \wedge^1 \R^3$. The norm on $\wedge^2\R^3$ is defined by setting $|v|_2=\langle v,v \rangle_2^{\frac{1}{2}}$ so that $|v_1\wedge v_2|_2$ is the 2- dimensional volume of the parallelipiped with vectors $v_1$ and $v_2$ as sides. The operator norm of the induced linear mapping $A^{\wedge 2}$ is
$$\norm{A^{\wedge 2}}_2 =\max \{|A^{\wedge 2}v|_2: |v|_2=1\}=\alpha_1(A)\alpha_2(A).$$

We say that $\A$ is 2-irreducible if $\forall 0 \neq v,w \in \wedge^2\R^3$, there exists $A \in \Se$ such that $\langle v,A^{\wedge 2} w \rangle_2 \neq 0$. It follows from \cite[Lemma 3.3]{kaenmaki2018structure} that $\A$ is 2-irreducible if and only if $\A$ is irreducible.

\subsubsection{Zariski topology} The Zariski topology on $\SL$ is defined by declaring closed sets to be sets of common zeros of collections of polynomial maps. This topology is finer than the usual topology on $\SL$. The Zariski closure of any subsemigroup of $\SL$ is a Lie group with finitely many connected components. We say that a subset $Z \subset \SL$ is a Zariski dense subset of $\SL$ if it is dense with respect to the Zariski topology, equivalently if  every polynomial which is identically zero on $Z$ is identically zero on $\SL$. Zariski density of $Z$ implies proximality and strong irreducibility of $Z$. 

\subsection{Singular value function}

For $A \in \SL$ let $\alpha_1(A) \geq \alpha_2(A) \geq \alpha_3(A)$ denote the singular values of $A$, noting that $\alpha_1(A)\alpha_2(A)\alpha_3(A)=1$ since $A \in \SL$. For $s \geq 0$, recall the definition of $\phi^s: \SL \to \R_{>0}$ from \eqref{svf}. When $A \in \SL_{>0}$, this is a projective analogue of Falconer's singular value function \cite{falconer1988hausdorff}, since $\frac{\alpha_2(A)}{\alpha_1(A)}$ and $\frac{\alpha_3(A)}{\alpha_1(A)}$ correspond to the larger and smaller contraction ratios for the action of $A$ on most of projective space. In \S \ref{sec:UB} we will see that this function can be used to describe a natural cover for $K_\A$.

\subsubsection{Multiplicativity properties}\label{s:mult} It will be important for us to understand how the multiplicativity properties of $\phi^s$ relate to the algebraic properties of our semigroup. We say that a function $\phi:\Se \to \R_{>0}$ is submultiplicative on $\Se$ if for all $A, B \in \Se$, $\phi(AB)\leq \phi(A)\phi(B)$ and we say $\phi$ is almost-submultiplicative if there exists $C< \infty$ such that for all $A, B \in \Se$, $\phi(AB)\leq C\phi(A)\phi(B)$. Note that this implies that the function $C\phi$ is submultiplicative.  If $s \in [0,1]$,
$$\phi^s(A)=\left(\frac{\alpha_2(A)}{\alpha_1(A)}\right)^s=\left(\frac{\alpha_1(A)\alpha_2(A)}{\alpha_1(A)^2}\right)^s =\left( \frac{\norm{A^{\wedge 2}}_2}{\norm{A}^2}\right)^s$$
and since
$$\frac{\alpha_3(A)}{\alpha_1(A)}=\frac{1}{\alpha_1(A) \cdot \alpha_1(A)\alpha_2(A)} = \frac{1}{\norm{A}\norm{A^{\wedge 2}}_2}$$
it follows that for $s \in [1,2]$ 
$$\phi^s(A)= \frac{\norm{A^{\wedge 2}}_2^{2-s}}{\norm{A}^{1+s}}.$$
\begin{prop}\label{submult}
Let $\A=\{A_i\}_{i \in \I} \subset \SL_{>0}$, where $\I$ is finite or countable. Suppose there exists $c>0$ such that for all $i \in \I$, $\frac{\min(A_i)_{j,k}}{\max(A_i)_{j,k}} \geq c$. 
% there exists a closed cone $K\subset \R^3_{>0}$ such that for each $A \in \Gamma$ maps $A(\R^3_{\geq 0}) \subset K$. 
Then
\begin{enumerate}
\item there exists $C>0$, $r \in (0,1)$ such that $\alpha_2(A_{\i|n})/\alpha_1(A_{\i|n}) \leq Cr^n$ for all $\i \in \I^n$,
\item for all $s \geq 0$, $\phi^s$ is almost-submultiplicative. Moreover the norm function $\norm{\cdot}$ is almost-supermultiplicative.
\end{enumerate}
\end{prop}

\begin{proof}
(2) follows from \cite[Lemma 7.1]{iommi2011almost}. In the finite case (1) is well known, see e.g. \cite{bochi2009some}. The infinite case is similar; the assumptions on the entries of the matrices ensure that the projective distance between the images of any two positive vectors $x,y$ under any projective map $A \in \A^n$ decreases exponentially at a uniform rate, and since $\frac{\alpha_2(A)}{\alpha_1(A)}$ times the projective distance between $x$ and $y$ is a lower bound on this distance, the claim is proved.
\end{proof}

We say that $\phi^s$ is quasimultiplicative on $\Se$ if there exists a finite set of words $\Gamma \subset \I^*$ and constant $c>0$ such that for all  $\i,\j \in \I^*$ there exists $\k \in \Gamma$ such that
$$\phi^s(A_{\i\k\j}) \geq c\phi^s(A_\i)\phi^s(A_\j).$$

The following lemma is similar to \cite[Lemma 3.5]{kaenmaki2018structure} and \cite[Proposition 2.8]{feng2009lyapunov}. 

\begin{lem}\label{quasilem}
If $\A$ is 2-irreducible then $\phi(A)=\norm{A^{\wedge 2}}_2$ is quasimultiplicative on $\Se$.
\end{lem}

\begin{proof}
For a contradiction assume that for all $n \in \N$ there exist $\i_n,\j_n \in \I^*$ such that
\begin{equation} \label{quasicontra}
\norm{A_{\i_n\k\j_n}^{\wedge 2}}_2 \leq \frac{\norm{A_{\i_n}^{\wedge 2}}_2 \norm{A_{\j_n}^{\wedge 2}}_2}{n}
\end{equation}
for all $\k \in \I^*$ with $|\k| \leq n$. Choose $u_n, v_n \in \wedge^2\R^3$ such that $|u_n|_2=|v_n|_2=1$, $\norm{A_{\i_n}^{\wedge 2}}_2=\norm{(A_{\i_n}^{\wedge 2})^Tu_n}_2$ and $\norm{A_{\j_n}^{\wedge 2}}_2=\norm{A_{\j_n}^{\wedge 2}v_n}_2$.

Setting $u_n'=\frac{(A_{\i_n}^{\wedge 2})^Tu_n}{\norm{A_{\i_n}^{\wedge 2}}_2}$ and $v_n'=\frac{A_{\j_n}^{\wedge 2}v_n}{\norm{A_{\j_n}^{\wedge 2}}_2}$, we see that Cauchy-Schwarz and \eqref{quasicontra} imply that for all $\k \in \I^*$ with $|\k| \leq n$,
$$\langle u_n', A_\k^{\wedge 2} v_n' \rangle_2= \frac{\langle (A_{\i_n}^{\wedge 2})^Tu_n , A_{\k\j_n}^{\wedge 2}v_n\rangle_2}{\norm{A_{\i_n}^{\wedge 2}}_2 \norm{A_{\j_n}^{\wedge 2}}_2} \leq \frac{\norm{A_{\i_n\k\j_n}^{\wedge 2}}_2}{\norm{A_{\i_n}^{\wedge 2}}_2 \norm{A_{\j_n}^{\wedge 2}}_2}  \leq \frac{1}{n}.$$
Let $u,v \in \wedge^2\R^3$ with $|u|_2=|v|_2=1$ be limit points of $u_n$ and $v_n$ respectively. Then $\langle u, A_\k^{\wedge 2}v \rangle_2=0$ for all $\k \in \I^*$, which contradicts 2-irreducibility of $\A$.

\end{proof}

\begin{cor}\label{quasiprop}
Suppose $\A$ is irreducible. Then for all $s \geq 0$, $\phi^s$ is quasimultiplicative on $\Se$.
\end{cor}

\begin{proof}
Since $\A$ is irreducible it is also 2-irreducible. The claim now follows directly from Lemma \ref{quasilem}, the expression for $\phi^s$ provided at the beginning of \S \ref{s:mult}, and the submultiplicativity of the norm.
\end{proof} 

\subsection{Pressure, zeta function and affinity dimension} Recall the definition of the zeta function and its critical exponent, the affinity dimension.
If $\zeta_\A(s)= \infty$ for all $s \geq 0$ we define $s_\A=\infty$. Since $\phi^s(A)$ is decreasing in $s$ for each $A \in \Se$, $\zeta_\A(s)=\infty$ for all $s<s_\A$ and $\zeta_\A(s)<\infty$ for all $s>s_\A$. The zeta function can be seen as a nonconformal, semigroup analogue of the Poincare series in the theory of limit sets of Fuchsian and Kleinian groups.

We can define another useful function when $\A \subset \SL_{>0}$ satisfies the assumptions of Proposition \ref{submult}. Define the pressure $P_\A:[0,\infty) \to \R$ by
\begin{equation} \label{pressure}
P_\A(s)=\lim_{n \to \infty} \frac{1}{n} \log \left( \sum_{A \in \A^n} \phi^s(A)\right),
\end{equation}
noting that the limit is well-defined by almost-submultiplicativity of $\phi^s$ (Proposition \ref{submult}(2)). $P_\A$ is a continuous, strictly decreasing, convex function in $s$. By the root test and Proposition \ref{submult}(1) it is easy to see that $s_\A$ is the unique value of $s$ for which $P_\A(s)=0$.

\subsection{Linear fractional transformations and the attractor}\label{lft}

From now until \S \ref{finalsection}, we restrict to finite sets $\A \subset \SL_{>0}$ which generate a Zariski dense semigroup $\Se$. Zariski density implies that $\Se$ is strongly irreducible. Instead of considering the limit set as a subset of $\R\P^2$ or $S$ as we did in the introduction, we instead construct it by considering the action of $\A$ on the plane $P=\{(x,y,1): x,y>0\}$. 

Given $A\in\mathrm{Mat}_{m,n}(\mathbb{R})$
we denote its rows by $r_{A,1},...,r_{A,m}\in\mathbb{R}^{n}$. For
$(x_{1},...,x_{n-1})=x\in\mathbb{R}^{n-1}$ set $\tilde{x}:=(x_{1},...,x_{n-1},1)$.
Denote by $\varphi_{A}$ the linear fractional transformation corresponding
to $A$:
\[
\varphi_{A}(x):=\left(\frac{\left\langle r_{A,1},\tilde{x}\right\rangle }{\left\langle r_{A,m},\tilde{x}\right\rangle },\frac{\left\langle r_{A,2},\tilde{x}\right\rangle }{\left\langle r_{A,m},\tilde{x}\right\rangle },...,\frac{\left\langle r_{A,m-1},\tilde{x}\right\rangle }{\left\langle r_{A,m},\tilde{x}\right\rangle }\right)\text{ for }x\in\mathbb{R}^{n-1}\text{ with }\left\langle r_{A,m},\tilde{x}\right\rangle \ne0,
\]
where $\left\langle \cdot,\cdot\right\rangle $ is the usual dot product. 

\begin{lem}
\label{lem:comp of frac lin trans}Let $d,m,n\ge2$, $A\in\mathrm{Mat}_{m,n}(\mathbb{R})$,
$B\in\mathrm{Mat}_{d,m}(\mathbb{R})$ and $x\in\mathbb{R}^{n-1}$
be given. Set $y:=\varphi_{A}(x)$, and suppose that $\left\langle r_{A,m},\tilde{x}\right\rangle \ne0$
and $\left\langle r_{B,d},\tilde{y}\right\rangle \ne0$. Then $\varphi_{B}\circ\varphi_{A}(x)=\varphi_{BA}(x)$.
\end{lem}

\begin{proof}
We have $\tilde{y}=\left(\frac{\left\langle r_{A,1},\tilde{x}\right\rangle }{\left\langle r_{A,m},\tilde{x}\right\rangle },...,\frac{\left\langle r_{A,m-1},\tilde{x}\right\rangle }{\left\langle r_{A,m},\tilde{x}\right\rangle },1\right)$.
Hence for $1\le j<d$,
\[
\frac{\left\langle r_{B,j},\tilde{y}\right\rangle }{\left\langle r_{B,d},\tilde{y}\right\rangle }=\frac{\left\langle r_{B,j},\left(\left\langle r_{A,1},\tilde{x}\right\rangle ,...,\left\langle r_{A,m},\tilde{x}\right\rangle \right)\right\rangle }{\left\langle r_{B,d},\left(\left\langle r_{A,1},\tilde{x}\right\rangle ,...,\left\langle r_{A,m},\tilde{x}\right\rangle \right)\right\rangle }=\frac{\left\langle r_{B,j},A\tilde{x}\right\rangle }{\left\langle r_{B,d},A\tilde{x}\right\rangle }=\frac{\left\langle r_{BA,j},\tilde{x}\right\rangle }{\left\langle r_{BA,d},\tilde{x}\right\rangle }.
\]
This implies,
\[
\varphi_{B}\circ\varphi_{A}(x)=\left(\frac{\left\langle r_{B,1},\tilde{y}\right\rangle }{\left\langle r_{B,d},\tilde{y}\right\rangle },...,\frac{\left\langle r_{B,d-1},\tilde{y}\right\rangle }{\left\langle r_{B,d},\tilde{y}\right\rangle }\right)=\left(\frac{\left\langle r_{BA,1},\tilde{x}\right\rangle }{\left\langle r_{BA,d},\tilde{x}\right\rangle },...,\frac{\left\langle r_{BA,d-1},\tilde{x}\right\rangle }{\left\langle r_{BA,d},\tilde{x}\right\rangle }\right)=\varphi_{BA}(x),
\]
which proves the lemma.
\end{proof}

We will be interested in two particular classes of linear fractional transformations. The first class are the maps induced by the action of each $A \in \Se$ on $P$, from which we will build the correponding IFS. Secondly, we will be interested in maps which are induced by projecting points in $P$ to other planes. 

\subsubsection{Induced projective maps} \label{induced}

Write $\mathbb{R}_{>0}^{2}:=\{(x_{1},x_{2})\in\mathbb{R}^{2}\::\:x_{1},x_{2}>0\}.$ For $A\in\mathrm{SL}(3,\mathbb{R})_{>0}$ and $x\in\mathbb{R}_{>0}^{2}$
we have $\varphi_{A}(x)\in\mathbb{R}_{>0}^{2}$. By Lemma \ref{lem:comp of frac lin trans}
this defines an action of $\mathrm{SL}(3,\mathbb{R})_{>0}$ on $\mathbb{R}_{>0}^{2}$. Let $F:\mathbb{R}_{>0}^{2}\rightarrow\mathrm{P}(\mathbb{R}^{3})$
be with $F(x):=\tilde{x}\mathbb{R}$ for $x\in\mathbb{R}_{>0}^{2}$.
Given $A\in\mathrm{SL}(3,\mathbb{R})_{>0}$ and $x\in\mathbb{R}_{>0}^{2}$,
\[
F(\varphi_{A}(x))=\mathrm{span}\left(\frac{\left\langle r_{A,1},\tilde{x}\right\rangle }{\left\langle r_{A,3},\tilde{x}\right\rangle },\frac{\left\langle r_{A,2},\tilde{x}\right\rangle }{\left\langle r_{A,3},\tilde{x}\right\rangle },1\right)=\mathrm{span}(A\tilde{x})=A(F(x)).
\]
Thus, the diffeomorphism $F$ is equivariant with respect to the actions
of $\mathrm{SL}(3,\mathbb{R})_{>0}$ on $\mathbb{R}_{>0}^{2}$ and
$F(\mathbb{R}_{>0}^{2})$.

From now on, our key object of study is the IFS $\{\varphi_A\}_{A \in \A}$ and its attractor, which with slight abuse of notation we also denote by $K_\A$. Note that this attractor is Lipschitz equivalent to the corresponding attractor contained in $\R\P^2$ or $S$, hence from the point of view of the Hausdorff dimension they are equivalent. 

Let $\Pi:\I^{\mathbb{N}}\rightarrow K_{\mathbf{A}}$ be the coding
map 
\[
\Pi(\i)=\underset{n\rightarrow\infty}{\lim}\varphi_{\i|_{n}}(x)\text{ for all }\i\in\I^{\mathbb{N}}\text{ and }x\in\mathbb{R}_{>0}^{2}.
\]

Later on we will be interested in choices of $\A$ for which the IFS $\{\varphi_A\}_{A \in \A}$ satisfies the following separation condition.
\begin{defn} \label{sosc}
We say that $\mathbf{A}$ satisfies the strong open set condition
(SOSC) if there exists an open set $U\subset\mathbb{R}_{>0}^{2}$
so that $U\cap K_{\mathbf{A}}\ne\emptyset$, $\cup_{i\in\I}\varphi_{i}(U)\subset U$
and $\varphi_{i}(U)\cap\varphi_{j}(U)=\emptyset$ for distinct $i,j\in\I$. 
\end{defn}

\subsubsection{Induced plane projections}

Write $\mathrm{Gr}_{2}(3)$ for the Grassmannian of $2$-dimensional
linear subspaces of $\mathbb{R}^{3}$.
Write $\mathbf{Y}$
for the set of $V\in\mathrm{Gr}_{2}(3)$ for which there exists $0\ne x\in V$
with nonnegative coordinates. Let $\mathbf{B}$ be the set of all $B\in\mathrm{Mat}_{2,3}(\mathbb{R})$
so that $|r_{B,2}|=1$, $r_{B,2}$ has nonnegative coordinates, and
$r_{B,1}$ and $r_{B,2}$ are linearly independent. Let $\mathbf{B}_{\mathrm{o}}$ be the set of all $B\in\mathbf{B}$
so that $\{r_{B,1},r_{B,2}\}$ is an orthonormal set. For $B\in\mathbf{B}$ we let $V_B \in \mathbf{Y}$ denote the plane $V_B=\mathrm{span}\{r_{B,1}, r_{B,2}\}$. 

For a plane $V \in \mathbf{Y}$ let $\pi_V: \R^3 \to V$ denote the orthogonal projection to $V$. For $B \in \mathbf{B}$, 
$$
\pi_{V_B}(\tilde{x})=\varphi_B(x)r_{B,1}+|r_{B,1}| r_{B,2}.$$

In particular, when $B \in \mathbf{B}_{\mathrm{o}}$, $\pi_{V_B}(\tilde{x})=\varphi_B(x)r_{B,1}+ r_{B,2}$.  We let $d_\P$ denote the metric on $\R\P^2$ induced by angles. It is easy to check that for any compact $K \subset \R^2_{>0}$, then
\begin{equation}d_{\P}(\pi_{V_B}(\tilde{x}),\pi_{V_B}( \tilde{y}))=\Theta_K(|\varphi_B(x)-\varphi_B(y)|) \label{planeequiv}
\end{equation}
for all $x,y \in K$ and $B \in \mathbf{B}_{\mathrm{o}}$. The consequence of \eqref{planeequiv} is that for any $V \in \mathbf{Y}$ and measure $\mu$ supported on $K_\A$, the dimension of $\varphi_B \mu$ is independent of the choice of representative $B \in \mathbf{B}_{\mathrm{o}}$ for which $V_B=V,$ and both of these concepts are equivalent to projecting $F(\mu)$ to the plane $V$.

\subsection{Stationary measures, Lyapunov dimension and reduction of Theorem \ref{msrthm}} \label{Slyap}

 Let $p=(p_{i})_{i\in\I}$ be a probability vector with strictly
positive coordinates. Since $\mathbf{S}$ is strongly irreducible
and proximal, there exists a unique $\mu\in\mathcal{M}(\mathbb{R}_{>0}^{2})$
with $\mu=\sum_{i\in\I}p_{i}\cdot\varphi_{i}\mu$. Here $\mathcal{M}(\mathbb{R}_{>0}^{2})$
is the space of compactly supported Borel probability measures on
$\mathbb{R}_{>0}^{2}$. Note that $\mathrm{supp}(\mu)=K_{\mathbf{A}}$
and $\mu=\Pi\beta$, where $\beta:=p^{\mathbb{N}}$ is the Bernoulli
measure corresponding to $p$. We will prove Theorem \ref{msrthm} for this measure $\mu$.

Later in this section we will see that Theorem \ref{msrthm} can be reduced to studying the dimension of the projection of $\mu$ to ``typical'' planes. ``Typical'' will be with respect to the following measure which is supported on $\mathrm{Gr}_{2}(3)$.  Since $\mathbf{S}$ is strongly
irreducible and proximal, there exists a unique $\nu\in\mathcal{M}(\mathrm{Gr}_{2}(3))$
so that $\nu=\sum_{i\in\I}p_{i}\cdot A_{i}^{*}\nu$.  Since $\{A_{i}^{*}\}_{i\in\I}\subset\mathrm{SL}(3,\mathbb{R})_{>0}$,
it is not difficult to see that $\nu(\mathbf{Y})=1$. For each $V\in\mathbf{Y}$
fix $B_{V}\in\mathbf{B}_{\mathrm{o}}$ so that $\mathrm{span}\{r_{B_{V},1},r_{B_{V},2}\}=V$. We equip $\mathrm{Gr}_{2}(3)$ with a metric $d$ induced from the metric $d_{\mathbb{P}}$ on $\R\P^2$.

Write $A$ in its singular value decomposition $A=VDU$ where $D$ is the diagonal matrix with $\alpha_1(A), \alpha_2(A), \alpha_3(A)$ down the diagonal. Given $A \in \Se$ such that $\alpha_2(A)>\alpha_3(A)$ we define $L_3(A)=Ve_3$. Define $L_3(\i)$ for $\i \in \I^*$ by $L_3(\i)=L_3(A_\i)$. We also extend the definition of $L_3$ to $\I^{\mathbb{N}}$, by
$$L_3(\i)= \lim_{n \to \infty} L_3(A_{\i|n}),$$
noting that this limit is defined $\beta$-almost everywhere and that $L_3\beta$ is the unique stationary measure on $\R\P^2$ associated to the probability vector $(p_i)_{i \in \I}$ and the set of inverses $\{A_i^{-1}\}_{i \in \I}$\footnote{Unlike $\mu$, the stationary measure $L_3\beta$ cannot be seen as a measure on $\R^2_{>0}$ since the inverses $\{A_i^{-1}\}_{i \in \I}$ are not positive.} (see \cite[\S2.11]{rapaport2022SelfAffRd} and references therein).

\subsubsection{Lyapunov dimension}

 Let $\chi_{1}(p)\ge\chi_{2}(p)\ge\chi_{3}(p)$ be the Lyapunov exponents
corresponding to $\mathbf{A}$ and $p$. That is,
\[
\chi_{i}(p)=\underset{n\rightarrow\infty}{\lim}\:\frac{1}{n}\log\alpha_{i}(A_{\i|_{n}})\text{ for }1\le i\le3\text{ and }\beta\text{-a.e. }\i.
\]
Since $\mathbf{S}$ is strongly irreducible and proximal, we in fact
have $\chi_{1}(p)>\chi_{2}(p)>\chi_{3}(p)$. Given a general ergodic measure $m$ on $\I^\N$ we denote the Lyapunov exponents corresponding to $\mathbf{A}$ and $m$ by $\chi_{1}(m)\ge\chi_{2}(m)\ge\chi_{3}(m)$  i.e. $\chi_i(p)\equiv \chi_i(\beta)$.

Let $H(p)=-\sum_{i \in\I} p_i\log p_i$ be the Shannon entropy of $p$. We also denote the measure theoretic entropy of an invariant measure $m$ on $\I^\N$ by $H(m)$, so that $H(\beta)=H(p)$. The Lyapunov dimension of $\mu$ is defined as in \cite[Section 1.3]{rapaport2020exact}) to be

\begin{equation}
\ld \mu=
\begin{cases}
\frac{H(p)}{\chi_1(p)-\chi_2(p)} &\mbox{if } 0 \leq H(p) \leq \chi_1(p)-\chi_2(p) \\
1+\frac{H(p)-(\chi_1(p)-\chi_2(p))}{\chi_1(p)-\chi_3(p)} & \mbox{if } \chi_1(p)-\chi_2(p) \leq H(p) \leq 2\chi_1(p)-\chi_2(p)-\chi_3(p)\\
2 &\mbox{if } H(p) \geq 2\chi_1(p)-\chi_2(p)-\chi_3(p) \end{cases} 
\label{lyapdim}
\end{equation}

We say that a measure $\eta$ is exact-dimensional if for $\eta$ almost every $x \in \mathrm{supp}\, \eta$, the limit
$$\lim_{r \to 0} \frac{\log \eta(\textup{B}(x,r))}{\log r}$$
exists and is $\eta$ a.e. constant. This constant is called the exact dimension and coincides with the Hausdorff dimension of the measure $\eta$. 

In \cite{rapaport2020exact} Rapaport proved that $\mu$ is exact dimensional (in a more general setting than we consider in this paper). The following result is a particular case of \cite[Theorem 1.3]{rapaport2020exact}.
\begin{thm}\label{thm:follows from LY formula}
 The measure $\mu$ is exact dimensional.
Moreover there exists $\Delta\in[0,1]$ so that,
\begin{enumerate}
\item $\varphi_{B_{V}}\mu$ is exact dimensional with $\hd\varphi_{B_{V}}\mu=\Delta$
for $\nu$-a.e. $V$;
\item $\hd \mu=\ld \mu$ whenever $\A$ satisfies the SOSC
and $\Delta=\min\{1,\frac{H(p)}{\chi_{1}(p)-\chi_{2}(p)}\}$.
\end{enumerate}
\end{thm}

\begin{proof}
Let $\mathcal{P}$ be the partition of $\I^{\mathbb{N}}$ according
to the first coordinate. That is $\mathcal{P}:=\{[i]\}_{i\in\I}$.
Denote by $H_{\beta}(\mathcal{P}\mid\Pi^{-1}(\mathcal{B}))$ the conditional
entropy of $\mathcal{P}$ given $\Pi^{-1}(\mathcal{B})$ with respect
to $\beta$, where $\mathcal{B}$ is the Borel $\sigma$-algebra of
$\mathbb{R}^{2}$. Since $\mathbf{A}$ satisfies
the SOSC, $\Pi^{-1}(\Pi(\i))$ is a singleton for $\beta$- almost every $\i$ (see e.g.  \cite[Corollary 2.8]{BK})). Therefore $H_{\beta}(\mathcal{P}\mid\Pi^{-1}(\mathcal{B}))=0$ and the proof now follows from \cite[Theorem 1.3]{rapaport2020exact}. 
\end{proof}

Theorem \ref{thm:follows from LY formula} is the starting point for our analysis and reduces the problem to showing that under the SOSC, $\Delta=\min\{1,\frac{H(p)}{\chi_{1}(p)-\chi_{2}(p)}\}$. We will actually prove this under a much weaker assumption than the SOSC, which we introduce now. We denote the operator norm by $\Vert\cdot\Vert$.

\begin{defn}\label{ESC}
We say that $\mathbf{A}$ satisfies the Diophantine property, if there exists $\epsilon>0$
so that,
\[
\Vert A_{u_{1}}-A_{u_{2}}\Vert\ge\epsilon^{n}\text{ for all }n\ge1\text{ and distinct }u_{1},u_{2}\in\I^{n}.
\]
\end{defn}

The following sections are dedicated to proving the following.

\begin{thm}
\label{thm:ESC --> Delta =00003D correct val}Suppose that $\mathbf{A}$
satisfies the Diophantine property. Then $\Delta=\min\{1,\frac{H(p)}{\chi_{1}(p)-\chi_{2}(p)}\}$,
where $\Delta$ is defined in Theorem \ref{thm:follows from LY formula}.
\end{thm}

 Note that Theorem \ref{thm:follows from LY formula} combined with Theorem \ref{thm:ESC --> Delta =00003D correct val} proves Theorem \ref{msrthm}. This is because the SOSC implies the Diophantine property which can be seen by combining \cite[Proposition 2.4]{solomyak2021diophantine} with  \cite[\S 6.2]{BHR}.

\subsection{Random measures and entropy}

\subsubsection{Dyadic partitions of $\R$ and random component measures}

The level-$n$ partition of $\R$ is defined
$$\D_n=\left\{[\frac{k}{2^n},\frac{k+1}{2^n}): k \in \Z\right\}.$$
We write $\D_t=\D_{[t]}$ when $t \in \R$ is not an integer. For $x \in \R$, $n \in \N$ we denote the element of $\D_n$ that $x$ belongs to by $\D_n(x)$. 

Let $\eta$ be a probability measure on $\R$. Assuming $\eta(A)>0$ let $\eta_A$ denote the conditional measure $\eta_A=\frac{1}{\eta(A)} \eta|_A$. For each $n \in \N$ we define measure valued random variables $\eta_{x,n}=\eta_{\D_n(x)}$, i.e. given $D \in \D_n$, $\eta_{x,n}=\eta_D$ with probability $\eta(D)$. We call this an $n$th level component of $\eta$. When several components appear together such as $\eta_{x,n}, \eta_{y,n}$ we assume that $x$ and $y$ are chosen independently, unless stated otherwise. Sometimes $n$ is chosen randomly too, usually uniformly in some range. For example for $n_2 \geq n_1$ and some event $U$,
\begin{equation} \label{randomise}
\P_{n_1\leq i \leq n_2}(\eta_{x,i} \in U)=\frac{1}{n_2-n_1+1} \sum_{n=n_1}^{n_2} \P(\eta_{x,n} \in U).
\end{equation}
We write $\mathbb{E}$ and $\mathbb{E}_{n_1 \leq i \leq n_2}$ for the expected value with respect to $\P$ and $\P_{n_1\leq i \leq n_2}$, where $\mathbb{P}=\mathbb{P}_{i=1}$.

\subsubsection{Dyadic-like partitions of $\mathbf{B}$ and random component measures} \label{dyadiclike}

Write $\mathrm{Aff}(\mathbb{R})$ for the matrix representation of
the affine group of $\mathbb{R}$. That is,
\[
\mathrm{Aff}(\mathbb{R}):=\left\{ \left(\begin{array}{cc}
a & b\\
0 & 1
\end{array}\right)\::\:a,b\in\mathbb{R}\text{ and }a\ne0\right\} .
\]
Note that $AB\in\mathbf{B}$ for all $A\in\mathrm{Aff}(\mathbb{R})$
and $B\in\mathbf{B}$.
\begin{lem}
\label{lem:def =000026 prop of inv metric}There exists a metric $d_{\mathbf{B}}$
on $\mathbf{B}$ so that,
\begin{enumerate}
\item (invariance property) $d_{\mathbf{B}}(AB_{1},AB_{2})=d_{\mathbf{B}}(B_{1},B_{2})$ for all
$B_{1},B_{2}\in\mathbf{B}$ and $A\in\mathrm{Aff}(\mathbb{R})$;
\item the topology induced by $d_{\mathbf{B}}$ is equal to the subspace
topology induced by $\mathrm{Mat}_{2,3}(\mathbb{R})$;
\item for every compact $Q\subset\mathbf{B}$ there exists $C=C(Q)>1$ so
that for all $B_{1},B_{2}\in Q$,
\[
C^{-1}\left\Vert B_{1}-B_{2}\right\Vert \le d_{\mathbf{B}}(B_{1},B_{2})\le C\left\Vert B_{1}-B_{2}\right\Vert ;
\]
\item $(\mathbf{B},d_{\mathbf{B}})$ is a complete metric space, which has
the Heine--Borel property (i.e. closed and bounded subsets of $\mathbf{B}$
are compact).
\end{enumerate}
\end{lem}

\begin{proof}
This is a special case of \cite[Lemma 2.6]{rapaport2022SelfAffRd}.
%//TODO (this should follow from \cite[Chapter 3,  Proposition 3.16]{CE}).
\end{proof}
From now on, unless stated otherwise explicitly, all metric concepts
in $\mathbf{B}$ are with respect to $d_{\mathbf{B}}$. In particular,
for $\emptyset\ne E\subset\mathbf{B}$ we write $\mathrm{diam}(E)$
for the diameter of $E$ with respect to $d_{\mathbf{B}}$, and for
$M\in\mathbf{B}$ and $r>0$ we write $\textup{B}(M,r)$ for the closed ball
in $(\mathbf{B},d_{\mathbf{B}})$ with centre $M$ and radius $r$.
For each $R\ge0$ set,
\[
\mathbf{B}(R):=\left\{ B\in\mathbf{B}\::\:d_{\mathbf{B}}(B,\mathbf{B}_{\mathrm{o}})\le R\right\} .
\]
By Lemma \ref{lem:def =000026 prop of inv metric} and since $\mathbf{B}_{\mathrm{o}}$
is compact, it follows that $\textup{B}(M,r)$ and $\mathbf{B}(R)$ are compact.

By \cite[Remark 2.2]{KRS} there exists a sequence $\{\mathcal{D}_{n}^{\mathbf{B}}\}_{n\ge0}$
of Borel partitions of $\mathbf{B}$ so that,
\begin{enumerate}
\item $\mathcal{D}_{n+1}^{\mathbf{B}}$ refines $\mathcal{D}_{n}^{\mathbf{B}}$
for each $n\ge0$. That is, for each $D\in\mathcal{D}_{n+1}^{\mathbf{B}}$
there exists $D'\in\mathcal{D}_{n}^{\mathbf{B}}$ with $D\subset D'$;
\item There exists a constant $R>1$ so that for each $n\ge0$ and $D\in\mathcal{D}_{n}^{\mathbf{B}}$
there exists $M_{D}\in D$ with
\[
\textup{B}(M_{D},R^{-1}2^{-n})\subset D\subset \textup{B}(M_{D},R2^{-n}).
\]
\end{enumerate}
When there is no risk of confusion, we shall write $\mathcal{D}_{n}$
in place of $\mathcal{D}_{n}^{\mathbf{B}}$.

For a measure $\theta$ on $\mathbf{B}$, and $n \in \N$ we will also consider random component measures $\theta_{B,n}=\theta_{\D_n(B)}$, and we will sometimes randomise uniformly over $n$ analogously to \eqref{randomise}.

\subsubsection{Random cylinder measures} \label{randomcyl}

We will sometimes decompose the stationary measure $\mu$ into its cylinder measures, i.e. measures of the form $\varphi_\i \mu$ for $\i \in \I^*$. To this end we introduce two partitions of $\I^\N$ and associated random words.

 Define
$$\Psi_n=\left\{i_1 \ldots i_{m} \in \I^* \; : \; \frac{\alpha_2(A_{i_1 \ldots i_{m}})}{\alpha_1(A_{i_1\ldots i_{m}})} \leq 2^{-n}< \inf_{1 \leq \ell \leq m} \frac{\alpha_2(A_{i_1 \ldots i_{\ell-1}})}{\alpha_1(A_{i_1 \ldots i_{\ell-1}})} \right\}.$$
Let $I(n) \in \Psi_n$ be a random word chosen according to the probability vector $p$, i.e. $\mathbb{P}(I(n)=\i)=p_\i$ if $\i \in \Psi_n$. For $\i \in \I^\N$ we let $\Psi_n(\i) \in \Psi_n$ denote the unique finite word which is a truncation of $\i$. Note that by definition of singular values, $|\varphi_A(x)-\varphi_A(y)|=O( \frac{\alpha_2(A)}{\alpha_1(A)})$ for all $x, y \in K_\A$ and $A \in \Se$.

We randomise with respect to cylinder measures analogously to \eqref{randomise}, i.e. for an observable $F$,
$$\mathbb{E}_{n_1\leq i \leq n_2}(F(\varphi_{I(i)}\mu))=\frac{1}{n_2-n_1+1} \sum_{n=n_1}^{n_2} \mathbb{E}(F(\varphi_{I(i)}\mu)).$$

Next we define a second way of partitioning $\I^\N$, which this time depends on a choice of $B \in \mathbf{B}_{\mathrm{o}}$. For $B \in \mathbf{B}_{\mathrm{o}}$ we define
\begin{equation} \label{xi}
\Xi_n^B=\left\{\i \in \I^*\; : \; \frac{\norm{A_\i^*u_{A_\i,B}}}{\norm{A_\i^* r_{B,2}}} \leq 2^{-n}< \inf_{1 \leq m \leq |\i|-1}\frac{\norm{A_{\i|m}^*u_{A_{\i|m},B}}}{\norm{A_{\i|m}^* r_{B,2}}} \right\}
\end{equation}
where for any $A \in \Se$, $u_{A,B}$ is defined to be the unique unit vector (up to sign) in $V_B$ with the property that $A^*u_{A,B} \perp A^*r_{B,2}$. It is not immediate from the definition that $\Xi_n^B$ is a partition for each $n \in\N$, but this will follow from Lemma \ref{rescale} and \eqref{xilem1}. We define $J(n,B)$ to be a random element from $\Xi_n^B$ chosen according to $p$. 

\subsubsection{Entropy}

For a probability measure $\eta$ and $\D, \D'$ finite or countable partitions of the underlying probability space, we define the entropy with respect to $\D$ by
$$H(\eta, \D)=-\sum_{I \in \D} \eta(I)\log \eta(I)$$
and the entropy with respect to $\D$, conditioned on $\D'$ by
$$H(\eta, \D|\D')=H(\eta, \D \wedge \D')-H(\eta, \D')$$
where $\D \wedge \D'$ denotes the common refinement of the partitions.
Given two partitions $\D, \D'$ such that each atom of $\D$ intersects at most $C$ atoms of $\D'$ and vice versa,
\begin{equation} \label{commens}
|H(\eta, \D)-H(\eta, \D')|=O(C)
\end{equation}
The entropy is concave: for $\alpha \in (0,1)$ and probability measures $\eta_1,\eta_2$,
$$H(\alpha \eta_1+(1-\alpha) \eta_2, D) \geq \alpha H(\eta, \D) +(1-\alpha) H(\eta_2, \D).$$
Given a measure $\eta$ on $\R^d$, if $f,g$ are such that $\sup_x |f(x)-g(x)|<2^{-n}$ then
\begin{equation}
|H(f\eta,\D_n)-H(g\eta, \D_n)|=O(1).
\label{almostcts}
\end{equation}
Using the definition of the distribution on components, given a measue $\eta$ on either $\R$ or $\mathbf{B}$ it follows that for $n,m \geq 1$,
\begin{equation}
H(\eta, \D_{n+m}|\D_n)= \E(H(\eta_{x,n},\D_{n+m})),\label{compentropy}
\end{equation}
where $\D_n$ is understood to be $\D_n^{\mathbf{B}}$ in the case that $\eta$ is a measure on $\mathbf{B}$.
By \eqref{planeequiv}, for any $B,B' \in \mathbf{B}_{\mathrm{o}}$ such that $V_B=V_{B'}$, $H(\varphi_{B} \mu, \mathcal{D}_n)=H(\varphi_{B'} \mu, \mathcal{D}_n)+O_K(1).$ We will use this fact implicitly throughout \S \ref{entropylb}.\\
Given $A=\begin{pmatrix}a&b\\0&1 \end{pmatrix} \in \mathrm{Aff}(\R)$, a measure $\eta$ on $\R$ and a measure $\theta$ on $\mathbf{B}$,
\begin{align}
\label{affineting}
H(\varphi_A\eta, \D_n) = H(\eta, \D_{n-\log a})&+O(1)= H(\eta, \D_n)+O(|\log a|+1) \\
|H(\theta, \D_n^{\mathbf{B}})&-H(A(\theta), \D_n^{\mathbf{B}})|=O(1) \label{affinetingB}
\end{align}
where we have used Lemma \ref{lem:def =000026 prop of inv metric} in \eqref{affinetingB}.

\section{Entropy lower bounds} \label{entropylb}

In this section we begin the work towards proving Theorem \ref{thm:ESC --> Delta =00003D correct val}. It is sufficient to assume $\Delta<1$. The proof will be by contradiction thus we will assume that $\Delta<\frac{H(p)}{\chi_1(p)-\chi_2(p)}$. Given any $n \in \N$ and $V \in \mathbf{Y}$, $\varphi_{B_V}\mu$ can be expressed as the expectation of $(\theta_n^{B_V})_{B,i} . \mu$, where $\theta_n^{B_V}= \sum_{\i \in \Xi_n^{B_V}} p_\i \delta_{B_VA_\i}$ and $\theta. \mu$ denotes the pushforward of $\theta \times \mu$ via the map $(M,x) \mapsto \varphi_M(x)$. Therefore $\Delta$, which equals the entropy of a typical projected measure $\varphi_{B_V} \mu$, can be bounded from below by the expected entropy of $(\theta_n^{B_V})_{B,i} . \mu$ by concavity of entropy. By relating $(\theta_n^{B_V})_{B,i} . \mu$ to a (random) linear convolution of measures, we will see that its entropy cannot be smaller than $\Delta$. Therefore a contradiction will be obtained by establishing the following entropy growth result (Theorem \ref{thm:ent inc under conv}): with positive probability $(\theta_n^{B_V})_{B,i} . \mu$ will have entropy significantly greater than $\Delta$.

In order to establish this, we will see that the Diophantine property combined with the assumption $\Delta < \frac{H(p)}{\chi_1(p)-\chi_2(p)}$ implies that there exists some $C>0$ such that with positive probability $(\theta_n^{B_V})_{B,Cn}$ has non-negligible entropy. Our main result will be to show that this, combined with the Zariski density implies the entropy growth result. The proof of this entropy growth result Theorem \ref{thm:ent inc under conv} will constitute the bulk of the following two sections. 

The entropy growth result will be proved by contradiction. A key element to this will be in showing that if $\Delta<1$ then most projected components have close to full entropy, which is the main result of this section. 
\begin{prop}
\label{prop:lb on ent of comp of mu} Suppose $\Delta<1$. For every $\epsilon,R>0$, $k\ge K(\epsilon,R)\ge1$,
$n\ge N(\epsilon,R,k)\ge1$ and $B\in\mathbf{B}(R)$,
\[
\mathbb{P}_{1\le i\le n}\left(\frac{1}{k}H(\varphi_{B}\mu_{x,i},\mathcal{D}_{i+k})>\Delta-\epsilon\right)>1-\epsilon.
\]
\end{prop}

Roughly speaking, Proposition \ref{prop:lb on ent of comp of mu} will be proved initially for $B=B_V \in \mathbf{B}_{\mathrm{o}}$ by realising $\varphi_{B_V} \mu$ as the expectation of rescaled copies of $\varphi_{B_{A_{I(n)}^*V}} \mu$ and showing that most projected measures $\varphi_{B_{A_{I(n)}^*V}} \mu$ have entropy close to $\Delta$. The result for $B \in \mathbf{B}(R)$ will follow by using the invariance of the metric $d_{\mathbf{B}}$.

\subsection{Expressing $\varphi_{B_VA}$ as an affine rescaling of $\varphi_{B_{A^*V}}$ }% change name

In this section we show that given $A \in \Se$, the entropy of $\varphi_{B_VA}\mu$ can be expressed in terms of the entropy of $\varphi_{B_{A^*V}}\mu$. We begin by rewriting the function $\varphi_{B_VA}$ as an affine rescaling of $\varphi_{B_{A^*V}}$.

\begin{lem}\label{rescale}
For $B\in\mathbf{B}_{\mathrm{o}}$ and $A \in \mathrm{SL}(3,\mathbb{R})_{>0}$, we have
\begin{equation}
\varphi_{BA}(x)=c_{A,B}\varphi_M(x)+t_{A,B} \label{rescale_eqn}
\end{equation}
 where $M \in \mathbf{B}_{\mathrm{o}}$ satisfies $V_M=A^*V_B$ and 
\begin{align*}
c_{A,B}&=\Theta_{\mathbf{S}}\left( \frac{\norm{A^* u_{A,B}}}{\norm{A^* r_{B,2}}}\right)
%\frac{\norm{A^*r_{B,1}}}{\norm{A^*r_{B,2}}}\sin (d_\P (A^*r_{B,1},A^*r_{B,2})),\\
%t_{A,B}&=\frac{\langle A^*r_{B,1}, A^*r_{B,2}\rangle}{\norm{A^*r_{B,2}}^2}.
\end{align*}
where $u_{A,B}$ was defined below \eqref{xi} to be the unique unit vector (up to sign) in $V_B$ with the property that $A^*u_{A,B} \perp A^*r_{B,2}$. 
\end{lem}

\begin{proof}
Observe that
$$BA= \begin{pmatrix} A^* r_{B,1} \\ A^* r_{B,2} \end{pmatrix} \in \mathbf{B}.$$
Define
$$v_{A,B}=A^*r_{B,1}-\frac{\langle A^*r_{B,1}, A^*r_{B,2}\rangle}{\norm{A^*r_{B,2}}^2}A^*r_{B,2}$$
noting that $v_{A,B} \in V_{BA}= A^*V_B$ and $\langle v_{A,B}, A^*r_{B,2} \rangle=0$. 

Writing $A^{-*}=(A^*)^{-1}$ define
$$u_{A,B}= \frac{A^{-*}v_{A,B}}{\norm{A^{-*}v_{A,B}}},$$
noting that this definition coincides with the one below \eqref{xi} and define
$$M=M_{A,B}= \begin{pmatrix}\frac{v_{A,B}}{\norm{v_{A,B}}} \\ \frac{A^* r_{B,2}}{\norm{A^* r_{B,2}}}\end{pmatrix} \in \mathbf{B}_{\mathrm{o}}.$$
One can easily check that
\begin{equation} \label{u'}
\varphi_{BA}(x)=\frac{\norm{A^{-*}v_{A,B}}\norm{A^* u_{A,B}}}{\norm{A^* r_{B,2}}} \varphi_{M_{A,B}}(x)+\frac{\langle A^*r_{B,1}, A^*r_{B,2}\rangle}{\norm{A^*r_{B,2}}^2}.
\end{equation}
Thus to complete the proof it is sufficient to show that $\norm{A^{-*}v_{A,B}}=\Theta(1)$ uniformly over $A \in \mathbf{S}$ and $B \in \mathbf{Y}$. To see this, first observe that by positivity of $\A$ there exists $c$ depending only on $\mathbf{A}$ such that for any $A \in \mathbf{S}$,
\begin{equation} \label{uniform}
c\alpha_1(A) \leq  \norm{A^* r_{B,2}} \leq \alpha_1(A).
\end{equation}
This allows us to show that $\norm{A^{-*}v_{A,B}}$ is bounded above and below by uniform constants. Indeed writing
$$\norm{A^{-*}v_{A,B}}=\norm{r_{B,1}-\frac{\langle A^*r_{B,1}, A^*r_{B,2}\rangle}{\norm{A^*r_{B,2}}^2}r_{B,2}}$$
then $r_{B,1}$ and $r_{B,2}$ are unit vectors perpendicular to each other, $\norm{A^{-*}v_{A,B}} \geq 1$. On the other hand
\begin{align*}
\norm{A^{-*}v_{A,B}}& \leq 1+ \frac{|\langle A^*r_{B,1}, A^*r_{B,2}\rangle|}{\norm{A^*r_{B,2}}^2}\\
&=1+ \frac{\norm{A^* r_{B,1}}}{\norm{A^* r_{B,2}}} \cos (d_\P (A^*r_{B,1},A^*r_{B,2}))\\
&\leq 1+\frac{1}{c}
\end{align*}
by applying \eqref{uniform}, thus we are done.\end{proof}

Currently the expression for the scaling constant $c_{A,B}$ is not useful, so over the next few lemmas we obtain a more useful estimate on $c_{A,B}$ for ``most'' choices of $A=A_\i$.

Recall the definition of $L_3$ from \S \ref{Slyap}, which we think of as a random variable on $(\I^{\mathbb{N}}, \beta)$, and $\Psi_n$ and $I(n)$ from \S \ref{randomcyl}. 
%The following is an analogue of \cite[Lemma 3.2]{HR} 
The following asserts that for any $V \in \mathbf{Y}$, most of the time $L_3(A_{I(n)})$ will not lie too close to $V$.

\begin{lem}\label{measure-bound}
For every $\epsilon>0$ and $0<\theta \leq \theta(\epsilon)$, if $n \geq N(\epsilon, \theta)$, then for all $V \in \mathbf{Y}$,
$$\mathbb{P}\left(\min_{v \in V}d_{\mathbb{P}}(L_3(A_{I(n)}), v) >  \theta\right)>1-\epsilon.$$

\end{lem}

\begin{proof}
The sequence $\{L_3(A_{\Psi_n(\i)})\}_{n \geq 1}$ converges to $L_3(\i)$ for $\beta$ almost every $\i \in \I^{\mathbb{N}}$. By definition
$$\mathbb{P}(I(n)=\mathbf{j})= \beta(\{\i: \Psi_n(\i)=\mathbf{j}\}).$$
So $\{L_3(A_{I(n)})\}_{n \geq 1}$ converges in distribution to $L_3$. Let 
$$A(V,\theta)=\{u \in \R\P^2: \min_{v \in V}d_{\mathbb{P}}(u,v)\} \leq \theta\}.$$ By irreducibility of $\A$, $L_3\beta$ does not give positive mass to any plane. Since $L_3\beta(A(V,\theta))$ is continuous in $V$ and $\theta$ there exists $\theta(\epsilon)>0$ such that for all $0<\theta\leq \theta(\epsilon)$,
$$L_3\beta(A(V,2\theta))<\epsilon/2$$
for all $V \in \mathbf{Y}$. Hence for all $n \geq N(\epsilon,\theta)$ and for all $V \in \mathbf{Y}$,
$$\mathbb{P}(L_3(A_{I(n)}) \in A(V,\theta))<\epsilon$$
completing the proof.

\end{proof}

%The following is an analogue of \cite[Lemma 2.1]{HR} and 
In the following we obtain an estimate for the scaling factor $c_{A,B}$ under the assumption that we have control over the distance of $L_3(A)$ from $V_B$. 

\begin{lem}\label{scaling}
Let $\epsilon>0$. There exist $c_\epsilon<C$, depending on $\mathbf{A}$, such that for any $A \in \mathbf{S}$ with $\min_{v \in V} d_{\mathbb{P}}(L_3(A),v) \geq \epsilon,$ 
\begin{equation} \label{xilem1}
c_\epsilon\frac{\alpha_2(A)}{\alpha_1(A)} \leq c_{A,B_V} \leq C \frac{\alpha_2(A)}{\alpha_1(A)}.
\end{equation}
Moreover, there exists $c>0$ such that for any $A \in \Se$ and $A' \in \A$,
\begin{equation} \label{xilem2}
\frac{\norm{(AA')^*u_{AA',B}}}{\norm{(AA')^* r_{B,2}}}\geq c\frac{\norm{A^*u_{A,B}}}{\norm{A^* r_{B,2}}}.
\end{equation}
\end{lem}

\begin{proof}

By \eqref{uniform}, in order to prove \eqref{xilem1} it is enough to show that
$$c_\epsilon\alpha_2(A) \leq \norm{A^*u_{A,B_V}} \leq C\alpha_2(A).$$

For the upper bound, consider the ellipse $A^* V \cap A^* ( \textup{B}(0,1))$. Let $x \in V$, $\norm{x}=1$ such that $A^*x$ is the longer semi-axes of this ellipse $E$ noting that $\norm{A^*x}\leq \alpha_1(A)$ and also that $x$ is positive since the longest semiaxes of the ellipsoid $A^*(\textup{B}(0,1))$ is positive and $V \in \mathbf{Y}$. Let $y \in V$, $\norm{y}=1$ so that $A^*y$ is the shorter semi-axes of the ellipse $E$, noting that $\norm{A^*y} \leq \alpha_2(A)$ since $A^*V$ must intersect the plane spanned by the two shorter semiaxes of $A^* ( \textup{B}(0,1))$ . 
Write $\theta=d_\P (A^*r_{B_V,2}, A^*x)$. Then
\begin{align*}
\norm{A^*u_{A,B_V}}&= \sqrt{ \norm{A^*u_{A,B_V}}^2\cos^2 \theta+ \norm{A^*u_{A,B_V}}^2 \sin^2 \theta}\\
&\leq \sqrt{\norm{A^*y}^2+ \norm{A^*x}^2 \sin^2 \theta}\\
&\leq \sqrt{\alpha_2(A)^2+C^2\alpha_1(A)^2 \frac{\alpha_2(A)^2}{\alpha_1(A)^2}}\\
&\leq \sqrt{1+C^2} \alpha_2(A)
\end{align*}
where in  the second line we have used that $A^*r_{B_V,2}$ and $A^*u_{A,B_V}$ are orthogonal to each other, and in the third line we have used that $\sin \theta \leq \theta$ and the fact that there exists $C=C(\mathbf{A})$, such that for all non-negative vectors $w,w'$, $d_\P(A^*w,A^*w')  \leq C \frac{\alpha_2(A)}{\alpha_1(A)}$.  This completes the proof of the upper bound.

We now prove the lower bound. Let $P_1$ be the plane spanned by the two longer semiaxes of the ellipsoid $A(\textup{B}(0,1))$ and $P_2$ be the plane spanned by $V^\perp$ (the normal direction to the plane $V$) and $u_{A,B_V}$. Let $\norm{x'}=1$ be such that $Ax' \in P_1 \cap P_2$. Since $Ax' \in P_1$, $\norm{Ax'} \geq \alpha_2(A)$. Also
$$d_{\mathbb{P}}(Ax',u_{A,B_V})=\frac{\pi}{2}-d_{\mathbb{P}}(Ax',V^\perp) \leq \frac{\pi}{2}-\epsilon$$
since $\min_{v \in V} d_{\mathbb{P}}(L_3(A),v) \geq \epsilon$. 
Now, 
\begin{align*}
\norm{A^*u_{A,B_V}} &= \sup_{\norm{v}=1} \langle v,A^*u_{A,B_V} \rangle \\
&= \sup_{\norm{v}=1} \langle Av, u_{A,B_V}\rangle \\
& \geq \langle Ax', u_{A,B_V}\rangle= \norm{Ax'} \cos (d_\P(Ax',u_{A,B_V})) \geq \alpha_2(A) \cos(\frac{\pi}{2}-\epsilon)
\end{align*}
which completes the proof with $c_\epsilon=\cos(\frac{\pi}{2}-\epsilon)$.

Next we turn to proving \eqref{xilem2}. Since  $\norm{A^* r_{B,2}} \leq \alpha_1(A)$ and $\alpha_1(\cdot)=\norm{\cdot}$ is submultiplicative, by \eqref{uniform} it is enough to prove that for all $A \in \Se$ and $A' \in \A$, 
$$\norm{(A A')^* u_{A A',B}} \geq c\norm{A^* u_{A,B}}.$$
This holds because for any $v \in A^*V \cap A^*(\textup{B}(0,1))$, $\norm{(A')^*v} \geq \alpha_3(A')\norm{v} \geq \alpha_3(A') \norm{A^*y} \geq \Theta_\A(1) \norm{A^* u_{A,B}}$ by invoking the bound $\theta \leq C\frac{\alpha_2(A)}{\alpha_1(A)}$.

%Let $\alpha_1(A_\i^*V)$ and $\alpha_2(A_\i^*V)$ denote the lengths of longer and shorter semiaxes of the ellipse $A_\i^*V \cap A_\i^*(B(0,1))$ respectively, where the longer semiaxes is denoted by $x_\i$ and the shorter semiaxes is denoted by $y_\i$. Let $\theta(\i)$ denote the angle between $A_\i^* r_{B,2}$ and $x_\i$.

%By Proposition \ref{submult}(2) there exists $c>0$ such that $\norm{A_\i^* r_{B,2}} \geq c\alpha_1(A_\i)$. This implies that the angle $\theta(\i) \leq C\frac{\alpha_2(A_\i)}{\alpha_1(A_\i)}$, which in turn implies that $\norm{A_\i^* u_{A_\i,B}} \leq C' \alpha_2(A_\i)$, where the constants $C$ and $C'$ depend only on $\A$. This proves \eqref{xilem1}.

\end{proof}

Combining Lemmas \ref{measure-bound} and \ref{scaling} gives:

\begin{prop}\label{lyapdiff}
For $\beta \times \nu$ almost every $(\i,V)$,
$$\lim_{n \to \infty}\frac{1}{n} \log c_{A_{\i|n},B_V}=\chi_2(p)-\chi_1(p).$$

\end{prop}

Finally we are able to reinterpret the entropy of most projected cylinder measures $\varphi_{B_VA_\i}\mu$ in terms of the entropy of $\varphi_{B_{A_\i^*V}}\mu$.

\begin{lem}\label{entropy_rescale2}
For every $\epsilon>0$ and $\theta>0$, if $m>M(\epsilon, \theta)$ the following holds for all $n \geq 1$: \;

For all $V \in \mathbf{Y}$ and $\mathbf{i} \in \Psi_n$ satisfying $\min_{v \in V} d_{\mathbb{P}}(L_3(A_\i),v) \geq \theta$, %$d(L_2(A_\i), V) < \frac{\pi}{2}- \theta$,
$$\left|\frac{1}{m} H(\varphi_{B_V A_\i }\mu, \mathcal{D}_{n+m})-\frac{1}{m} H(\varphi_{B_{A_\i^* V}} \mu, \mathcal{D}_m)\right|<\varepsilon.$$

\end{lem}

\begin{proof}
Let $\i \in \Psi_n$. By Lemmas \ref{rescale} and \ref{scaling} it follows that $\min_{v \in V} d_{\mathbb{P}}(L_3(A_\i),v) \geq \theta$ implies
$$H(\varphi_{B_V A_\i} \mu, \mathcal{D}_m)= H(\varphi_{B_{A_\i^*V}} \mu, \mathcal{D}_{m+ \log \left(\frac{\alpha_2(A_\i)}{\alpha_1(A_\i)}\right)})+O_\theta(1)=H(\varphi_{B_{A_\i^*V}} \mu, \mathcal{D}_{m+ n})+O_\theta(1)$$
where the final line follows by definition of $\Psi_n$.
\end{proof}

\subsection{The $\mu$-measure of neighbourhoods of dyadic grids}

Given $\emptyset\ne E\subset\mathbb{R}^{2}$ and $\delta>0$ we denote
by $E^{(\delta)}$ the closed $\delta$-neighbourhood of $E$. For
the proof of Proposition \ref{prop:lb on ent of comp of mu} we shall
need the following lemma.

\begin{lem}
\label{lem:mu mass of bd of dyad}
For every $\epsilon>0$ there exists
$\delta>0$ so that for all $n\ge0$,
\[
\mu\left(\cup_{D\in\mathcal{D}_{n}^{2}}(\partial D)^{(2^{-n}\delta)}\right)<\epsilon.
\]
\end{lem}

Setting
\[
B_{1}:=\left(\begin{array}{ccc}
1 & 0 & 0\\
0 & 0 & 1
\end{array}\right)\text{ and }B_{2}:=\left(\begin{array}{ccc}
0 & 1 & 0\\
0 & 0 & 1
\end{array}\right),
\]
we have $\varphi_{B_{1}}(x)=x_{1}$ and $\varphi_{B_{2}}(x)=x_{2}$
for $(x_{1},x_{2})=x\in\mathbb{R}_{>0}^{2}$. Thus, Lemma \ref{lem:mu mass of bd of dyad}
follows easily from the following lemma.

\begin{lem}\label{thickening}
For every $\epsilon,R>0$ there exists $\delta=\delta(\epsilon,R)>0$
so that for every $B\in\mathbf{B}_{\mathrm{o}}$,
\[
\varphi_{B}\mu(t-\delta r,t+\delta r)\le\epsilon\cdot\varphi_{B}\mu(t-r,t+r)\text{ for all }t\in\mathbb{R}\text{ and }0<r<1.
\]
\end{lem}

\begin{proof}
Without loss of generality we can assume that $\mu$ is supported on $\textup{B}(0,1)$. We begin by proving that for all $\delta>0$ there exists $\rho>0$ such that for all $B \in \mathbf{B}_{\mathrm{o}}$ and for all $x \in \R^2_{>0}$, 
$$\varphi_B\mu(\textup{B}(x,\rho)) <\delta.$$
For a contradiction we suppose the above is false. Notice that $\{\varphi_B\mu\}_{B \in\mathbf{B}_{\mathrm{o}}}$ is compact in the weak-* topology. Thus there must exist $B \in \mathbf{B}_{\mathrm{o}}$ and $x \in \R^2_{>0}$ such that $\varphi_B\mu(\{x\})>0$, i.e. $\mu(\varphi_B^{-1}(\{x\}))>0$. But $\varphi_B^{-1}(\{x\})$ is a line in $\R^2$, hence this contradicts our irreducibility assumption.

We are now ready to prove the lemma. By Lemma \ref{scaling} there exists $c_0\in (0,1)$ such that
$$c_02^{-n} \leq \frac{\norm{A_\i^*u_{A_\i,B}}}{\norm{A_\i^* r_{B,2}}} \leq 2^{-n}$$
for all $n \in \N$ and $\i \in \Xi_n^B$. Let $r, \epsilon>0$ and choose $\delta>0$ such that
$$\sup_x \sup_{B \in \mathbf{B}_{\mathrm{o}}} \varphi_B\mu(\textup{B}(x, \frac{\delta}{c_0}))<\epsilon.$$
Choose $n=-\log (r/3)$. Now
$$\varphi_B \mu(\textup{B}(x, \frac{\delta r}{3}))= \mathbb{E}(\varphi_{BA_{J(n,B)}} \mu(\textup{B}(x, \frac{\delta r}{3})),$$
noting that 
\begin{align} \label{bitofmeasure}
\varphi_{BA_{J(n,B)}} \mu(\textup{B}(x, \frac{\delta r}{3}))&= T_{A,B}S_{A,B} \varphi_M  \mu(\textup{B}(x, \frac{\delta r}{3}))
\end{align}
where $T_{A,B}$ denotes translation by $t_{A_{J(n,B)},B}$, $S_{A,B}$ denotes scaling by $c_{A_{J(n,B)},B}$ and $M=M_{A,B}$ is as in Lemma \ref{rescale}. By definition of $J(n,B)$, $c_{A_{J(n,B)},B}=\frac{\norm{A_{J(n,B)}^*u_{A_{J(n,B)},B}}}{\norm{A_{J(n,B)}^* r_{B,2}}} \in (\frac{c_0r}{3}, \frac{r}{3}]$. Hence, by rescaling the expression in \eqref{bitofmeasure} by $1/c_{A_{J(n,B)},B}$ we obtain the mass of a ball of radius $<\delta/c_0$ with respect to $\varphi_M \mu$ which by choice of $\delta$ is less than $\epsilon$. However, this is only if the mass is positive, so by conditioning on that event we obtain
\begin{align*}
\mathbb{E}(\varphi_{BA_{J(n,B)}} \mu(\textup{B}(x, \frac{\delta r}{3}))&\leq \epsilon \mathbb{P}\left\{\varphi_{B}\varphi_{A_{J(n,B)}}(\textup{B}(0,1)) \cap \textup{B}(x,\frac{\delta r}{3}) \neq \emptyset \right\}\\
&\leq \epsilon \varphi_B \mu(\textup{B}(x,r))
\end{align*}
noting that this is where we have used the assumption that $\mu$ is supported on $\textup{B}(0,1)$.

\end{proof}

\subsection{Proof of Proposition \ref{prop:lb on ent of comp of mu}}

Throughout this section we will use the assumption that $\Delta<1$. In Lemma \ref{uniformlb}, by expressing $\varphi_{B_V}\mu$ as the expectation of rescaled copies of $\varphi_{B_{A_{I(n)}^*V}}\mu$ we will combine the fact that these equidistribute to $\nu$ and the fact that $\nu$ almost surely $\varphi_{B_V}\mu=\Delta$ to obtain a global lower bound on the entropy of $\varphi_{B_V}\mu$ which holds over all $V \in \mathbf{Y}$. This will imply almost immediately that for $B \in \mathbf{B}(R)$, most projected cylinder measures $\varphi_{BA_{I(n)}}\mu$ have entropy close to $\Delta$ in Lemma \ref{cylindermeasures}. By combining this with Lemma \ref{lem:mu mass of bd of dyad} we will be able to extend this to a  statement about most projected component measures, which will prove Proposition \ref{prop:lb on ent of comp of mu}.

\begin{lem}
\label{lem:ent(thet.sig)>=00003Davg(ent of comp)}
For $\eta \in \mathcal{M}(\R)$, let $n,m \geq 1$ and $k \geq 0$ be given. Suppose $\mathrm{diam}( \mathrm{supp} \, \eta)=O(2^{-k})$. Then
$$\frac{1}{n} H(\eta, \D_{k+n})=\E_{k \leq i \leq k+n} \left(\frac{1}{m} H(\eta_{x,i},\D_{i+m})\right)+O(\frac{m}{n}).$$
\end{lem}

\begin{proof}Similar to the proof of \cite[Lemma 2.5]{HR}.
\end{proof}

We begin by obtaining a global lower bound on the entropy of projections of $\mu$.

\begin{lem}\label{uniformlb}
For every $\epsilon>0$ and $n \geq N(\epsilon) \geq 1$,
$$\inf_{V \in \mathbf{Y}} \frac{1}{n} H(\varphi_{B_V}\mu, \mathcal{D}_n) >\Delta-\epsilon.$$
\end{lem}

\begin{proof}
This proof is very similar to \cite[Lemma 3.3]{HR}; here we describe how to adapt that proof to our current setting.
By Lemma \ref{lem:ent(thet.sig)>=00003Davg(ent of comp)}, \eqref{compentropy} and assuming $n$ is sufficiently large with respect to $m$, it follows that for $V \in \mathbf{Y}$,
\begin{equation}\label{uniformlbx}
\frac{1}{n} H(\varphi_{B_V} \mu, \D_n)=\frac{1}{n} \sum_{k=1}^n \frac{1}{m} H(\varphi_{B_V} \mu, \D_{k+m}|\D_k)+O(\epsilon).
\end{equation}
Next, note that $\varphi_{B_V} \mu=\mathbb{E}_{i=k}(\varphi_{B_V}\varphi_{A_{I(i)}}\mu)$ and $\mathrm{supp}(\varphi_{A_{I(i)}} \mu)=\Theta(2^{-i})$ by definition of $\Psi_i$. By these facts, Lemmas \ref{entropy_rescale2} and \ref{measure-bound} we can proceed from \eqref{uniformlbx} exactly as in the proof of \cite[Lemma 3.3]{HR} to deduce that
\begin{equation} \label{uniformlb1}
\frac{1}{n} H(\varphi_{B_V} \mu, \mathcal{D}_n) \geq \mathbb{E}_{1 \leq i \leq n}( \frac{1}{m} H(\varphi_{B_{A_{I(i)}^* V}} \mu, \mathcal{D}_m))-O(\epsilon).
\end{equation}
We now use the fact that (roughly speaking) $A_{I(n)}^* V$ equidistributes to $\nu$. Indeed, since $\hd \varphi_{B_V} \mu=\Delta$ for $\nu$ almost every $V \in \mathbf{Y}$, then if $m$ is large enough then
$$\frac{1}{m} H(\varphi_{B_V} \mu, \mathcal{D}_m)>\Delta-\frac{\epsilon}{2}$$
for all $V$ in a set of $\nu$ measure $>1-\epsilon$. 

Applying \eqref{almostcts} to the projections $\varphi_{B_V}$ implies that there exists an open set $U \subset \mathbf{Y}$ with $\nu(U)>1-\epsilon$ such that
$$\frac{1}{m} H(\varphi_{B_V}\mu, \mathcal{D}_m)>\Delta-\epsilon$$
for all $V \in U$.

We are almost done. Finally note that we have the following analogue of \cite[Proposition 2.8]{HR}: if $U \subset \mathbf{Y}$ is open and $\nu(U)>1-\epsilon$ then for $n \geq N(\epsilon, U)$,
$$\inf_{V \in \mathbf{Y}} \E_{1 \leq i \leq n}(\delta_{A^*_{I(i)}V}(U))>1-C\epsilon$$
for some $C=C(\A)$. The last two displayed equations imply that for all $n$ large relative to $\epsilon$,
$$\mathbb{P}_{1 \leq i \leq n}\left(\frac{1}{m} H(\varphi_{B_{A_{I(i)}^* V}} \mu, \mathcal{D}_m)>\Delta-\epsilon\right) \geq 1-O(\epsilon)$$
which combined with \eqref{uniformlb1} proves the result.
\end{proof}

This can now easily be used to obtain a lower bound on entropy of most projected cylinder measures:

\begin{lem} \label{cylindermeasures}
For every $\epsilon>0$, $m \geq M(\epsilon)$ and $n \geq N(\epsilon)$,
$$\inf_{B \in \mathbf{B}(R)} \mathbb{P} \left(\frac{1}{m} H(\varphi_{BA_{I(n)}} \mu, \mathcal{D}_{n+m}) \geq \Delta-\epsilon\right)>1-\epsilon.$$
\end{lem}

\begin{proof}
We begin by proving that this infimum holds over all $B \in \mathbf{B}_{\mathrm{o}}$. To see this notice that by Lemmas \ref{entropy_rescale2} and \ref{measure-bound}, it is enough to prove (possibly for a different $\epsilon$) that
\begin{equation}
\inf_{V \in \mathbf{Y}} \mathbb{P}\left(\frac{1}{m} H(\varphi_{B_{A_{I(n)}^* V}}\mu, \mathcal{D}_m) \geq \Delta-\epsilon\right)>1-\epsilon
\label{Bo}
\end{equation}
which follows from Lemma \ref{uniformlb}. Now we use \eqref{Bo} to prove the lemma. Note that if  $B=\begin{pmatrix}r_{B,1}\\ r_{B,2}\end{pmatrix} \in \mathbf{B}(R)$ then $B'=\begin{pmatrix}\frac{v}{\norm{v}}\\  r_{B,2}\end{pmatrix} \in \mathbf{B}_{\mathrm{o}}$ where $v=r_{B,1}-\langle r_{B,1},r_{B,2}\rangle r_{B,2}$. Similarly to the proof of Lemma \ref{rescale}, we see that $\varphi_B=\norm{v}\varphi_{B'}+\langle r_{B,1}, r_{B,2}\rangle$. Since $B \in \mathbf{B}(R)$, $\norm{v}=O_R(1)$. Therefore for any $A \in \Se$, by \eqref{affineting}
\begin{equation} \label{Br}
H(\varphi_{BA}\mu, \D_n) = H(\varphi_{B'A}\mu,\D_n)+O_R(1).
\end{equation}
Combining \eqref{Bo} and \eqref{Br} gives the lemma.
\end{proof}

Lemma \ref{cylindermeasures} is the type of result we want, except Proposition \ref{prop:lb on ent of comp of mu} is a statement about projected component measures rather than projected cylinder meaures. Lemma \ref{lem:mu mass of bd of dyad} provides the extra ingredient to make this step.

\begin{proof}[Proof of Proposition \ref{prop:lb on ent of comp of mu}]
Assume that $\Delta<1$. Let $\epsilon>0$, $\delta$ be small relative to $\epsilon$, $k \geq 1$ large with respect to $\delta$, $m$ large with respect to $k$ and $n$ large with respect to $\epsilon$. By Lemma \ref{lem:mu mass of bd of dyad} we may assume that
$$\mu\left(\cup_{D\in\mathcal{D}_{n}^{2}}(\partial D)^{(2^{-n}\delta)}\right)<\epsilon.$$
Since $k$ is large relative to $\delta$ we may assume that if $\eta$ is such that $\mathrm{diam}(\mathrm{supp}\, \eta) \leq C \cdot 2^{-n-k}$ and
$$\# \{D \in \mathcal{D}_n \; : \; \mathrm{supp}(\eta) \cap D \neq \emptyset\}>1,$$
then $\mathrm{supp}( \eta) \subseteq \cup_{D\in\mathcal{D}_{n}^{2}}(\partial D)^{(2^{-n}\delta)}$. 
Hence 
\begin{align*}
\mathbb{P}_{i=n+k}( \textnormal{$\varphi_{A_{I(i)}}\mu$ is contained in a level$-n$ dyadic cell})>1-\mu\left(\cup_{D\in\mathcal{D}_{n}^{2}}(\partial D)^{(2^{-n}\delta)}\right)>1-\epsilon.
\end{align*}
On the other hand by applying Lemma \ref{cylindermeasures} with $n+k$ instead of $n$ we obtain
$$\mathbb{P}_{i=n+k}(\frac{1}{m} H(\varphi_{BA_{I(i)}} \mu, \mathcal{D}_{i+m}) \geq \Delta-\epsilon)>1-\epsilon.$$
The last two displayed equations can now be combined exactly as in the proof of \cite[Lemma 3.7]{HR} to deduce that
$$\mathbb{P}_{i=n}\left( \frac{1}{m} H(\varphi_{B}\mu_{x,i}, \mathcal{D}_{i+m})>\Delta-O(\sqrt{\epsilon}\right)>1-O(\sqrt{\epsilon})$$
which is what we wanted (if we start from a smaller $\epsilon$).
\end{proof}

\subsection{Entropy of components of projections of components}

Apart from Proposition \ref{prop:lb on ent of comp of mu} we will also require the following finer statement. Given integers $n\ge k\ge1$ let $\mathcal{N}_{k,n}:=\{k,k+1,..,n\}$
and denote the normalized counting measure on $\mathcal{N}_{k,n}$
by $\lambda_{k,n}$, i.e. $\lambda_{k,n}\{i\}=\frac{1}{n-k+1}$ for
each $k\le i\le n$. We write $\mathcal{N}_{n}$ and $\lambda_{n}$
in place of $\mathcal{N}_{1,n}$ and $\lambda_{1,n}$.
\begin{lem}
\label{lem:lb on ent of comp of proj of comp}For every $\epsilon,R>0$,
$m\ge M(\epsilon,R)\ge1$, $k\ge K(\epsilon,R)\ge1$, $n\ge N(\epsilon,R,k,m)\ge1$
and $B\in\mathbf{B}(R)$,
\[
\int\mathbb{P}_{i\le j\le i+k}\left\{ \frac{1}{m}H\left((\varphi_{B}\mu_{x,i})_{y,j},\mathcal{D}_{j+m}\right)>\Delta-\epsilon\right\} \:d\lambda_{n}\times\mu(i,x)>1-\epsilon.
\]
\end{lem}

\begin{proof}
Let $\epsilon>0$ be small, let $R>0$, let $\delta>0$ be small with respect to $\epsilon,R$, let $q \geq 1$ be large with respect to $q$, let $m,k \geq 1$ be large with respect to $q$, let $n \geq 1$ be large with respect to $m,k$ and fix $B \in \mathbf{B}(R)$.

By Lemma \ref{lem:mu mass of bd of dyad},
\begin{align*}
\int\int \varphi_B \mu_{x,i}\left(\cup_{D\in\mathcal{D}_{j}^{2}}(\partial D)^{(2^{-j}\delta)}\right) &d \lambda_{i,i+k}(j)d \lambda_n \times \mu(i,x) \\&= \int \int \varphi_B \mu \left(\cup_{D\in\mathcal{D}_{j}^{2}}(\partial D)^{(2^{-j}\delta)}\right)d \lambda_{i,i+k}(j)d \lambda_n(i)<\epsilon. \end{align*}
By Proposition \ref{prop:lb on ent of comp of mu},
$$\int \P_{i \leq j \leq i+k} \left(\left\{ \frac{1}{m} H(\varphi_B(\mu_{x,i})_{y,j}, \D_{j+m})>\Delta-\epsilon \right\} \right) d \lambda_n \times \mu(i,x)>1-\frac{\epsilon}{2}-O(\frac{k}{n})>1-\epsilon$$
where we have used \cite[Lemma 2.7]{Ho} in the first inequality.

We can now proceed exactly as in \cite[Lemma 3.11]{rapaport2022SelfAffRd}) to deduce that there exists $Z \subset \mathcal{N}_n \times \R^2$ with $\lambda_n \times \mu(Z)>1-\epsilon$ such that for all $(i,x) \in Z$
\begin{equation} \label{rap1}
\P_{s=j+q} \left(\left\{ \frac{1}{m} H(\varphi_B(\mu_{x,i})_{y,s}, \D_{s-q+m})>\Delta-\epsilon \right\} \right) >1-\epsilon
\end{equation}
and for each $(i,x) \in Z$ there exists a set $Q_{i,x} \subset \{i, \ldots, i+k\}$ with $\lambda_{i,i+k}(Q_{i,x})>1-\epsilon$ such that for all $j \in Q_{i,x}$,
\begin{equation} \label{rap2}
\varphi_B \mu_{x,i}\left(\cup_{D\in\mathcal{D}_{j}^{2}}(\partial D)^{(2^{-j}\delta)}\right)<\epsilon.
\end{equation}
Fix $(i,x) \in Z$ and $j \in Q_{i,x}$. Since $B \in \mathbf{B}(R)$,
$$\mathrm{diam}(\mathrm{supp} (\varphi_B(\mu_{x,i})_D)=O_R(2^{-j-q})$$
for all $D \in \D_{j+q}^2$ with $\mu_{x,i}(D)>0$. We can now follow the proof exactly as in \cite[Lemma 3.11]{rapaport2022SelfAffRd}) to show how this, \eqref{rap1}, \eqref{rap2} can be combined with
$$\E_{s=j}((\varphi_B \mu_{x,i})_{y,s})=\varphi_B \mu_{x,i}=\E_{s=j+q}(\varphi_B(\mu_{x,i})_{y,s})$$
and the concavity of entropy to deduce that
$$\P_{s=j} \left(\left\{ \frac{1}{m} H(\varphi_B(\mu_{x,i})_{y,s}, \D_{s+m})>\Delta-\epsilon \right\} \right) >1-\epsilon$$
for all $(i,x) \in Z$ and $j \in Q_{i,x}$. Since $\lambda_n \times \mu(Z)>1-\epsilon$ and $\lambda_{i,i+k}(Q_{i,x})>1-\epsilon$ we are done.
\end{proof}

\section{An entropy increase result}

In this section, we will use Proposition \ref{prop:lb on ent of comp of mu} to prove the aforementioned entropy increase result. This will be divided into two results, one which deals with $\theta.\mu$ for $\theta$ supported close to $\mathbf{B}_{\mathrm{o}}$ i.e. projections of matrices which are close to conformal (Theorem \ref{thm:ent inc under conv}) and one which deals with the other case (Theorem \ref{thm:ent inc under conv2}), which will follow from the former by the invariance of the metric $d_{\mathbf{B}}$.

Given $\theta\in\mathcal{M}(\mathbf{B})$ and $\sigma\in\mathcal{M}(\mathbb{R}_{>0}^{2})$, recall that we
write $\theta.\sigma\in\mathcal{M}(\mathbb{R})$ for the pushforward
of $\theta\times\sigma$ via the map sending $(B,x)$ to $\varphi_{B}(x)$.

\begin{thm}\label{thm:ent inc under conv2}
Suppose $\Delta<1$ and let $\epsilon,R>0$ be given. Then there exists $\delta=\delta(\epsilon,R)>0$ so that for all $n\ge N(\epsilon,R,\delta)\ge1$ and  $\theta \in \mathcal{M}(\mathbf{B})$ for which
\begin{itemize}
\item[ (a)] there exists $B \in \mathbf{B}_{\mathrm{o}}$ and $A \in \mathbf{S}$ such that $BA \in \mathrm{supp}\,\theta$ and $\mathrm{supp}\, \theta \subset \textup{B}(BA, R)$, 
\item [ (b)] $\frac{1}{n} H(\theta, \mathcal{D}_n) \geq \epsilon$,
\end{itemize}we have $\frac{1}{n}H(\theta.\mu,\mathcal{D}_{n+ \log c_{A,B}})>\Delta+\delta$.
\end{thm}

Theorem \ref{thm:ent inc under conv2} will follow from the following theorem.

\begin{thm}
\label{thm:ent inc under conv}Suppose that $\Delta<1$ and let $\epsilon,R>0$
be given. Then there exists $\delta=\delta(\epsilon,R)>0$ so that
for all $n\ge N(\epsilon,R,\delta)\ge1$ and $\theta\in\mathcal{M}(\mathbf{B}(R))$
with $\frac{1}{n}H(\theta,\mathcal{D}_{n})\ge\epsilon$, we have $\frac{1}{n}H(\theta.\mu,\mathcal{D}_{n})>\Delta+\delta$.
\end{thm}

Let us elucidate on the sense in which Theorems \ref{thm:ent inc under conv2} and \ref{thm:ent inc under conv} are considered entropy \emph{increase} results. To see this, we begin with the following linearisation result. Given $\theta\in\mathcal{M}(\mathbf{B})$ and $x\in\mathbb{R}_{>0}^{2}$
we write $\theta.x$ in place of $\theta.\delta_{x}$. Let $*$ denote convolution of measures on $\R$.
\begin{lem}
\label{lem:linearisation}Let $Z\subset\mathbf{B}\times\mathbb{R}_{>0}^{2}$
be compact. Then for every $\epsilon>0$, $k\ge K(\epsilon)\ge1$,
and $0<\delta<\delta(Z,\epsilon,k)$ the following holds. Let $(M_{0},x_{0})\in Z$,
$\theta\in\mathcal{M}(\textup{B}(M_{0},\delta))$, and $\sigma\in\mathcal{M}(\textup{B}(x_{0},\delta)\cap\mathbb{R}_{>0}^{2})$
be given. Then,
\begin{equation}
\left|\frac{1}{k}H\left(\theta.\sigma,\mathcal{D}_{k-\log\delta}\right)-\frac{1}{k}H\left((\theta.x_{0})*(\varphi_{M_{0}}\sigma),\mathcal{D}_{k-\log\delta}\right)\right|<\epsilon.
\label{lineareqn}
\end{equation}
\end{lem}

\begin{proof} Throughout this proof we write $x=(u,v) \in \R^2_{>0}$ and we identify $M=\begin{pmatrix}a&b&c\\d&e&f \end{pmatrix} \in \mathbf{B}$ with $(a,b,c,d,e,f) \in\R^3 \times  S^2_{\geq 0} $ where 
 $S^2_{ \geq 0}=\{(a,b,c) \in \R^3, a,b,c \geq 0, a^2+b^2+c^2=1\}$. Consider the map $g: \R^3 \times S^2  \times \R^2_{>0} \to \R$ given by
$$g(a,b,c,d,e,f,u,v)= \frac{au+bv+c}{du+ev+f}.$$
By considering $\theta$ as a measure on $S^2 \times \R^3$ then $g(\theta \times \sigma)=\theta.\sigma$. Without loss of generality we can assume that for all $M \in \textup{B}(M_0,\delta)$ and $x \in (x_0,\delta)$, $|au+bv+c| \geq 1$. Indeed, if this is not satisfied, by simply ``translating'' $\theta$ and $M_0$ in the hypothesis to $A\theta$ and $AM_0$ where $A=\begin{pmatrix}1&2\\0&1 \end{pmatrix}$, we recover the assumption and only change the expression in \eqref{lineareqn} by a factor $O(1/k)$ by \eqref{affineting}. 

A computation shows that the first order approximation of $g$ at \\$z_0=(M_0,x_0)=(a_0,b_0,c_0,d_0,e_0,f_0,u_0,v_0)$, evaluated at $z=(M,x)=(a,b,c,d,e,f,u,v)$ is given by
\begin{align*}
G_{z_0}(z)&=g(z_0)+\nabla g_{z_0}(z-z_0) \\
&=\varphi_{M_0}(x)+\varphi_M(x_0)-\varphi_{M_0}(x_0)  \\
& + (\varphi_{M_0}(x)-\varphi_M(x_0)) \frac{d_0(u-u_0)+e_0(v-v_0)}{a_0u_0+b_0v_0+c_0} \\
&+ (\varphi_M(x_0)-\varphi_{M_0}(x_0)) \frac{(d-d_0)u_0+(e-e_0)v_0+f-f_0}{a_0u_0+b_0v_0+c_0}.
\end{align*}
By compactness of $Z$, for $z,z_0 \in Z$ with $M \in \textup{B}(M_0,\delta)$ and $x \in \textup{B}(x_0,\delta)$, the last two lines can be bounded  above, in absolute value by $O_Z(\delta^2)$.  Write $h_{z_0}(z)=\varphi_{M_0}(x)+\varphi_M(x_0)-\varphi_{M_0}(x_0) $ so that $(h_{z_0}) (\theta \times \sigma)=\delta_{-\varphi_{M_0}(x_0)} *(\theta.x_0)*(\varphi_{M_0}\sigma)$. Again using compactness of $Z$, for all $k \geq 1$ there exists $\delta(Z,k)>0$ such that for any $0<\delta \leq \delta(Z,k)$ and $z_0 \in Z$,
\begin{align*}
\norm{g-h_{z_0}}_{\textup{B}(M_0,\delta) \times \textup{B}(x_0,\delta)} \leq \norm{g-G_{z_0}}_{\textup{B}(M_0,\delta) \times \textup{B}(x_0,\delta)} +\norm{G_{z_0}-h_{z_0}}_{\textup{B}(M_0,\delta) \times \textup{B}(x_0,\delta)}\leq \delta 2^{-k}
\end{align*}
where $\norm{\cdot}_{\textup{B}(M_0,\delta) \times \textup{B}(x_0,\delta)}$ denotes the supremum norm on $\textup{B}(M_0,\delta) \times \textup{B}(x_0,\delta)$. By \eqref{almostcts}
\begin{align*}
H(\theta.\sigma, \mathcal{D}_{k-\log \delta})=H(g(\theta \times \sigma), \mathcal{D}_{k-\log \delta})&=H((h_{z_0})(\theta \times \sigma), \mathcal{D}_{k-\log \delta})+O_Z(1) \\
&= H((\theta.x_0)*(\varphi_{M_0}\sigma), \mathcal{D}_{k-\log \delta})+O_Z(1).
\end{align*}
Since $O_Z(1/k)$ can be made less than $\epsilon$ by making $k$ large, we are done.
\end{proof}

Lemma \ref{lem:linearisation} tells us that locally near $(B_0,x_0)$, $\theta.\mu$ can be approximated by the Euclidean convolution $\theta.x_0*\varphi_{B_0}\mu$. The results in the previous section tell us that the entropy of $\varphi_{B_0}\mu$ should not be much smaller than $\Delta$. Since Euclidean convolution is a smoothing operation, this implies that the entropy of $\theta.x_0*\varphi_{B_0}\mu$ should not be much smaller than $\Delta$. In particular $\Delta$ is a `trivial' lower bound on $\frac{1}{n}H(\theta.\mu,\mathcal{D}_{n})$, and so heuristically Theorems \ref{thm:ent inc under conv2} and \ref{thm:ent inc under conv} are giving conditions under which $\frac{1}{n}H(\theta.\mu,\mathcal{D}_{n})$ grows significantly above this trivial lower bound.

%The plan for this section is as follows. We will use Proposition \ref{prop:lb on ent of comp of mu} to show that if a component $\theta$ has non-negligible entropy but the entropy of the convolution $\theta. \mu$ does not significantly grow, {\color{red}then this can be rephrased in terms of statements in terms of the ``concentration'' of these measures, Proposition \ref{prop:key prop for ent inc}. We make precise what we mean by concentration in the following preparatory section. Theorem \ref{thm:ent inc under conv} will then be proved by contradiction, by showing that being able to do this for arbitrarily small parameters is incompatible with the Zariski density. }

Theorem \ref{thm:ent inc under conv} will be proved by contradiction. In \S \ref{atom} we introduce a notion of ``atomicity'' which will be important throughout this section. In \S \ref{entropyatom} we assume that Theorem \ref{thm:ent inc under conv} does not hold for some parameters $\epsilon$ and $\delta$  and obtain a translation of this using the language of ``atomicity'' of appropriate measures. In \S \ref{incproof} this is used to show that failure of Theorem \ref{thm:ent inc under conv}  is incompatible with the Zariski density of $\Se$.

\subsection{Atomicity} \label{atom}

Statements which describe how close the entropy of a measure is to zero are naturally related to statements which describe how ``concentrated'' or ``close to atomic'' this measure are and it is by translating the hypothesis of Theorem \ref{thm:ent inc under conv} into this language that will allow us to glean the algebraic implications of our hypothesis. This motivates the following definition.

\begin{defn}
Let $(X,d)$ be a metric space. Given $i\ge0$, $\epsilon>0$ and
$\theta\in\mathcal{M}(X)$, we say that $\theta$ is $(i,\epsilon)$-atomic
if $\theta(\textup{B}(x,2^{-i}\epsilon))>1-\epsilon$ for some $x\in X$.
\end{defn}

We begin by stating a couple of technical results we'll require concerning $(i,\epsilon)$-atomicity. These closely resemble existing results from the literature, which are highlighted.

\begin{lem}
\label{lem:from conc on R to cont on B}For every $\epsilon,R>0$
there exist $\delta=\delta(\epsilon,R)>0$ and $b=b(\epsilon,R)\ge1$
so that the following holds. Let $\theta\in\mathcal{M}(\mathbf{B}(R))$,
$x\in K_{\mathbf{A}}$ and $k\ge1$ be with
\[
\mathbb{P}_{i=k}\left((\theta.x)_{y,i}\text{ is }(i,\delta)\text{-atomic}\right)>1-\delta.
\]
Then,
\[
\mathbb{P}_{k\le i\le k+b}\left(\theta_{B,i}\ldotp x\text{ is }(i,\epsilon)\text{-atomic}\right)>1-\epsilon.
\]
\end{lem}

\begin{proof}
Similar to the proof of \cite[Proposition 5.5]{Ho}. Consider the identities
$$\E_{i=k}((\theta. x)_{y,i})=\theta=\E_{i=k}(\theta_{B,i}. x).$$
So $\theta_{B,i} . x$ are ($\theta$ almost surely over the choice of $B$) absolutely continuous with respect to the weighted average of the components $(\theta . x)_{y,i}$. Moreover, by Lemma \ref{lem:def =000026 prop of inv metric} and the definition of the partitions in \S \ref{dyadiclike}, each $\theta_{B,i} . x$ is supported on a set which intersects $m=O_R(1)$ level $k$ dyadic cells $(\theta . x)_{y,i}$. Hence each ``component'' $\theta_{B,i} . x$ is absolutely continuous with respect to an average of these $O_R(1)$ components $(\theta . x)_{y,i}$. 

By a Markov inequality argument as in \cite[Proposition 5.5]{Ho}, this implies
\begin{equation}
\mathbb{P}_{i=k}\left( \theta_{B,i}.x \textnormal{  is $(i,\delta')^m$- atomic}\right)>1-\delta'
\end{equation}
where $\delta' \to 0$ as $\delta \to 0$ and we write that a measure $\eta$ on $\R$ is $(k,\delta')^m$-atomic to mean that there exists $m' \leq m$ and $x_1, \ldots, x_{m'} \in \R$ such that $\eta \left(\bigcup_{j=1}^{m'} \textup{B}(x_j, 2^{-k}\delta')\right) \geq 1-\delta'.$ The proof then follows as in \cite[Proposition 5.5]{Ho} to deduce that for $\delta''=O_{R,m}(\log \log(1/\delta')/\log(1/\delta'))$ and $b=\frac{1}{2} \log (1/\delta')$,
$$\mathbb{P}_{k\le i\le k+b}\left(\theta_{B,i}\ldotp x\text{ is }(i,\delta'')\text{-atomic}\right)>1-\delta''$$
where $\delta''$ can be made arbitrarily small by taking $\delta$ small.
\end{proof}
Recall that $\mathbf{B}\subset\mathrm{Mat}_{2,3}(\mathbb{R})$ and
so $\Vert\cdot\Vert$ induces a metric on $\mathbf{B}$.
\begin{lem}
\label{lem:comp not conc}For every $\epsilon,R>0$ there exists $\delta=\delta(\epsilon,R)>0$
such that for $k\ge K(\epsilon,R,\delta)\ge1$ and $n\ge N(\epsilon,R,\delta,k)\ge1$
the following holds. Let $\theta\in\mathcal{M}(\mathbf{B}(R))$ be
with $\frac{1}{n}H(\theta,\mathcal{D}_{n})>\epsilon$. Then $\lambda_{n}\times\theta(F)>\delta$,
where $F$ is the set of all $(i,B_{1})\in\mathcal{N}_{n}\times\mathbf{B}$
such that
\[
\mathbb{P}_{i\le j\le i+k}\left\{ (\theta_{B_{1},i})_{B_{2},j}\text{ is not }(j,\delta)\text{-atomic with respect to }\Vert\cdot\Vert\right\} >\delta.
\]
\end{lem}

\begin{proof}
The proof is similar to \cite[Lemma 5.9]{rapaport2022SelfAffRd}). Let $C>1$ be a large global constant, $\epsilon, R>0$, $\ell \geq 1$ be large with respect to $\epsilon,R$, $k \geq 1$ large with respect to $\ell$, $n \geq 1$ large with respect to $k$ and $\theta \in \mathcal{M}(\mathbf{B}(R))$ be such that $\frac{1}{n}H(\theta, \mathcal{D}_n)>\epsilon.$

Using a similar argument as in \cite[Lemma 6.8]{HR}, we may assume that $\lambda_n \times \theta(F)>\frac{\epsilon}{C}$ where $F'$ is the set of $(i,B_1) \in \mathcal{N}_n \times \mathbf{B}$ such that
$$\mathbb{P}_{i \leq j \leq i+k}\left(\left\{\frac{1}{\ell} H((\theta_{B_1,i})_{B_2,j}, \mathcal{D}_{\ell+j})>\frac{\epsilon}{C}\right\}\right)>\frac{\epsilon}{C}.$$
Let $(i,B_1) \in \mathcal{N}_n \times \mathbf{B}$, $(j,B_2) \in \mathcal{N}_{i,i+k} \times \mathbf{B}$ be such that $\sigma=(\theta_{B_1,i})_{B_2,j}$ is well-defined and $\frac{1}{\ell} H(\sigma, \mathcal{D}_{\ell +j})>\frac{\epsilon}{C}$. It suffices to show that $\sigma$ is not $(j,2^{-\ell})$ atomic with respect to $\norm{\cdot}$.

Assume for a contradiction that $\sigma(B')>1-2^{-\ell}$ for some closed ball $B'$ of $\norm{\cdot}$- radius $2^{-j-\ell}$. Since $\mathrm{supp}(\sigma)$ is contained in a ball of radius $O(R)$ around $\mathbf{B}_{\mathrm{o}}$ and by Lemma \ref{lem:def =000026 prop of inv metric}, we have $\sigma(B)>1-2^{-\ell}$ for some closed ball $B$ of $d_{\mathbf{B}}$- radius $O(2^{-j-\ell})$. Since $\mathrm{diam}(B)=O(2^{-j-\ell})$ and by definition of the partitions in \S \ref{dyadiclike}, 
$$\# \{D \in \mathcal{D}_{j+\ell}^{\mathbf{B}}: D \cap B \neq \emptyset\}=O(1)$$
which implies $H(\sigma_B, \mathcal{D}_{j+\ell})=O(1).$ Since $\sigma$ is supported on some $D \in \mathcal{D}_j^{\mathbf{B}}$ 
$$\frac{1}{\ell} H(\sigma_{B^c}, \mathcal{D}_{j+\ell})=O(1).$$
Thus by convexity of entropy,
\begin{align*}
\frac{\epsilon}{C}< \frac{1}{\ell} H(\sigma, \mathcal{D}_{j+\ell}) \leq \frac{1}{\ell}(\sigma(B)H(\sigma_B, \mathcal{D}_{j+\ell})+\sigma(B^c) H(\sigma_{B^c}, \mathcal{D}_{j+\ell})+O(1)=O(1/\ell).
\end{align*}
Since $\ell$ is large with respect to $\epsilon$, this is a contradiction.
\end{proof}

We will also require the following variation on Lemma \ref{lem:ent(thet.sig)>=00003Davg(ent of comp)}.

\begin{lem}\label{part2}
Let $R>0$ and $\theta\in\mathcal{M}(\mathbf{B}(R))$
be given. Then for every $n\ge k\ge1$,
\[
\frac{1}{n}H(\theta.\mu,\mathcal{D}_{n})\ge\mathbb{E}_{1\le i\le n}\left(\frac{1}{k}H(\theta_{B,i}.\mu_{x,i},\mathcal{D}_{i+k})\right)+O_{R}\left(\frac{1}{k}+\frac{k}{n}\right).
\]
\end{lem}

\begin{proof}
Similar to the proof of \cite[Lemma 6.9]{HR}.
\end{proof}

\subsection{From entropy to atomicity}\label{entropyatom}

In the following key proposition we show that if we \emph{don't} have entropy increase, i.e. $\frac{1}{n}H(\theta, \mathcal{D}_n)\geq \epsilon$ and $\frac{1}{n}H(\theta.\mu,\mathcal{D}_{n})\le\Delta+\delta$ then we can find a measure $\xi$ on $\mathbf{B}(R)$ which is not very concentrated in measure, but such that for a high proportion of $x \in \R^2$, $\xi . x$ is very concentrated in measure.
\begin{prop}
\label{prop:key prop for ent inc}Suppose that $\Delta<1$ and let
$\epsilon,R>0$ be given. Then there exists $\epsilon'=\epsilon'(\epsilon,R)>0$
such that for all $\sigma>0$ there exists $\delta=\delta(\epsilon,R,\sigma)>0$
so that for all $n\ge N(\epsilon,R,\sigma)\ge1$ the following holds.
Let $\theta\in\mathcal{M}(\mathbf{B}(R))$ be such that $\frac{1}{n}H(\theta,\mathcal{D}_{n})\ge\epsilon$
and $\frac{1}{n}H(\theta.\mu,\mathcal{D}_{n})\le\Delta+\delta$. Then
there exist $\xi\in\mathcal{M}(\mathbf{B}(R))$ and $j\ge0$ so that,
\begin{enumerate}
\item $\mathrm{diam}(\mathrm{supp}(\xi))=O_{R}(2^{-j})$ with respect to
$\Vert\cdot\Vert$;
\item $\xi$ is not $(j,\epsilon')$-atomic with respect to $\Vert\cdot\Vert$;
\item $\mu\left\{ x\::\:\xi.x\text{ is }(j,\sigma)\text{-atomic}\right\} >1-\sigma$;
\end{enumerate}
\end{prop}

\begin{proof}
Let $\delta>0$ be small with respect to $\epsilon,R$, let $k\ge1$
be large with respect to $\delta$, let $n\ge1$ be large with respect
to $k$, and let $\theta\in\mathcal{M}(\mathbf{B}(R))$ be with $\frac{1}{n}H(\theta,\mathcal{D}_{n})\ge\epsilon$
and $\frac{1}{n}H(\theta.\mu,\mathcal{D}_{n})\le\Delta+\delta^{2}$. 

By Lemma \ref{part2},
\[
\Delta+2\delta^{2}\ge\mathbb{E}_{1\le i\le n}\left(\frac{1}{k}H(\theta_{B,i}.\mu_{x,i},\mathcal{D}_{i+k})\right).
\]
From this and Lemma \ref{lem:linearisation},
\begin{equation}
\Delta+3\delta^{2}\ge\mathbb{E}_{1\le i\le n}\left(\frac{1}{k}H(\theta_{B,i}.x*\varphi_{B}\mu_{x,i},\mathcal{D}_{i+k})\right).\label{eq:ub on avg of conv of proj}
\end{equation}

Write $\Gamma:=\lambda_{n}\times\mu\times\theta$ and let $E_{0}$
be the set of all $(i,x,B)\in\mathcal{N}_{n}\times\mathbb{R}_{>0}^{2}\times\mathbf{B}$
so that,
\[
\frac{1}{k}H(\theta_{B,i}.x*\varphi_{B}\mu_{x,i},\mathcal{D}_{i+k})<\frac{1}{k}H(\varphi_{B}\mu_{x,i},\mathcal{D}_{i+k})+\delta.
\]
By \cite[Equation (2.17)]{rapaport2022SelfAffRd} we may assume that
for $\Gamma$-a.e. $(i,x,B)$,
\begin{equation}
\frac{1}{k}H(\theta_{B,i}.x*\varphi_{B}\mu_{x,i},\mathcal{D}_{i+k})\ge\frac{1}{k}H(\varphi_{B}\mu_{x,i},\mathcal{D}_{i+k})-\delta^{2}.\label{eq:conv dont dec ent}
\end{equation}
From this, by the definition of $E_{0}$, and from (\ref{eq:ub on avg of conv of proj}),
\[
\Delta+4\delta^{2}\ge\int\mathbb{E}_{1\le i\le n}\left(\frac{1}{k}H(\varphi_{B}\mu_{x,i},\mathcal{D}_{i+k})\right)\:d\theta(B)+\delta\Gamma(E_{0}^{c}).
\]
Moreover, by Proposition \ref{prop:lb on ent of comp of mu}
\[
\int\mathbb{E}_{1\le i\le n}\left(\frac{1}{k}H(\varphi_{B}\mu_{x,i},\mathcal{D}_{i+k})\right)\:d\theta(B)\ge\Delta-\delta^{2}.
\]
Thus, by replacing $\delta$ with a larger quantity (which is still
small with respect to $\epsilon,R$) without changing the notation,
we get $\Gamma(E_{0})\ge1-\delta$.

Let $\sigma>0$ be small with respect to $\epsilon,R$ and suppose
that $\delta$ is small with respect to $\sigma$. Let $m\ge1$ be
an integer which is large with respect to $\sigma$ and suppose that
$\delta$ is small with respect to $m$. By \cite[Theorem 2.7]{Ho1},
it follows that for every $(i,x,B)\in E_{0}$
\begin{equation}
\mathbb{P}_{i\le j\le i+k}\left\{ \begin{array}{c}
\frac{1}{m}H\left((\varphi_{B}\mu_{x,i})_{y,j},\mathcal{D}_{j+m}\right)>1-\sigma\text{ or }\\
(\theta_{B,i}.x)_{z,j}\text{ is }(j,\sigma)\text{-atomic}
\end{array}\right\} >1-\sigma.\label{eq:from Hoch thm in R}
\end{equation}
\begin{lem*}
We can assume that $\Gamma(E_{1})>1-\sigma$, where $E_{1}$ is the
set of all $(i,x,B)\in\mathcal{N}_{n}\times\mathbb{R}_{>0}^{2}\times\mathbf{B}$
with
\[
\mathbb{P}_{i\le j\le i+k}\left\{ (\theta_{B,i}.x)_{z,j}\text{ is }(j,\sigma)\text{-atomic}\right\} >1-\sigma.
\]
\end{lem*}
\begin{proof}
Let $Z$ be the set of all $(i,x,B)\in\mathcal{N}_{n}\times\mathbb{R}_{>0}^{2}\times\mathbf{B}$
so that,
\[
\mathbb{P}_{i\le j\le i+k}\left\{ \left|\frac{1}{m}H\left((\varphi_{B}\mu_{x,i})_{y,j},\mathcal{D}_{j+m}\right)-\Delta\right|<\sigma\right\} >1-\sigma.
\]
Let $\sigma'>0$ be small with respect to $\sigma$ and suppose that
$m$ is large with respect to $\sigma'$. From (\ref{eq:ub on avg of conv of proj})
and (\ref{eq:conv dont dec ent}),
\[
\Delta+\delta\ge\int\mathbb{E}_{1\le i\le n}\left(\frac{1}{k}H(\varphi_{B}\mu_{x,i},\mathcal{D}_{i+k})\right)\:d\theta(B).
\]
Thus, by \cite[Lemma 2.9]{rapaport2022SelfAffRd},
\[
\Delta+\sigma'\ge\int\mathbb{E}_{i\le j\le i+k}\left(\frac{1}{m}H\left((\varphi_{B}\mu_{x,i})_{y,j},\mathcal{D}_{j+m}\right)\right)\:d\Gamma(i,x,B).
\]
Moreover, by Lemma \ref{lem:lb on ent of comp of proj of comp},
\[
\int\mathbb{P}_{i\le j\le i+k}\left\{ \frac{1}{m}H\left((\varphi_{B}\mu_{x,i})_{y,j},\mathcal{D}_{j+m}\right)>\Delta-\sigma'\right\} \:d\Gamma(i,x,B)>1-\sigma'.
\]
From the last two inequalities, and since $\sigma'$ is small with
respect to $\sigma$, we obtain $\Gamma(Z)>1-\sigma$. The lemma now
follows from this and $\Gamma(E_{0})\ge1-\delta$, by (\ref{eq:from Hoch thm in R}),
since $\sigma$ may be assumed to be small with respect to $1-\Delta>0$,
and by replacing $\sigma$ with a larger quantity (which is still
small with respect to $\epsilon,R$) without changing the notation.
\end{proof}
\begin{lem*}
We can assume that $\Gamma(E_{2})>1-\sigma$, where $E_{2}$ is the
set of all $(i,x,B_{1})\in\mathcal{N}_{n}\times\mathbb{R}_{>0}^{2}\times\mathbf{B}$
with
\begin{equation}
\mathbb{P}_{i\le j\le i+k}\left\{ (\theta_{B_{1},i})_{B_{2},j}.x\text{ is }(j,\sigma)\text{-atomic}\right\} >1-\sigma.\label{eq:def prop of E_2}
\end{equation}
\end{lem*}
\begin{proof}
Fix $(i,x,B_{1})\in E_{1}$ with $x\in K_{\mathbf{A}}$, and set
\[
S:=\left\{ j\in\mathcal{N}_{i,i+k}\::\:\mathbb{P}_{l=j}\left\{ (\theta_{B_{1},i}.x)_{y,l}\text{ is }(l,\sigma)\text{-atomic}\right\} \ge1-\sqrt{\sigma}\right\} .
\]
By the definition of $E_{1}$ it follows that $\lambda_{i,i+k}(S)\ge1-\sqrt{\sigma}$.
Let $\sigma'>0$ be small with respect to $\epsilon,R$ and suppose
that $\sigma$ is small with respect to $\sigma'$. By Lemma \ref{lem:from conc on R to cont on B},
there exists an integer $b=b(R,\sigma')\ge1$ such that,
\begin{equation}
\mathbb{P}_{j\le l\le j+b}\left((\theta_{B_{1},i})_{B_{2},l}.x\text{ is }(l,\sigma')\text{-atomic}\right)\ge1-\sigma'\text{ for }j\in S.\label{eq:conc for j in S}
\end{equation}

Let $\sigma''>0$ be small with respect to $\epsilon,R$ and suppose
that $\sigma'$ is small with respect to $\sigma''$. From $\lambda_{i,i+k}(S)\ge1-\sqrt{\sigma}$
and (\ref{eq:conc for j in S}), by assuming that $\sigma,\sigma'$
are sufficiently small with respect to $\sigma''$, and by assuming
that $k$ is sufficiently large with respect to $b$, it follows by
a statement similar to \cite[Lemma 2.7]{Ho} that (\ref{eq:def prop of E_2})
is satisfied with $\sigma''$ in place of $\sigma$. This completes
the proof of the lemma.
\end{proof}
By the previous lemma, by Fubini's theorem, and by replacing $\sigma$
with a larger quantity without changing the notation, we may assume
that $\lambda_{n}\times\theta(F_{1})>1-\sigma$, where $F_{1}$ is
the set of all $(i,B_{1})\in\mathcal{N}_{n}\times\mathbf{B}$ such
that
\[
\mathbb{P}_{i\le j\le i+k}\left(\mu\left\{ x\::\:(\theta_{B_{1},i})_{B_{2},j}.x\text{ is }(j,\sigma)\text{-atomic}\right\} >1-\sigma\right)>1-\sigma.
\]

Let $\epsilon'>0$ be small with respect to $\epsilon,R>0$, and suppose
that $\sigma$ is small with respect to $\epsilon'$. By Lemma \ref{lem:comp not conc},
and the assumptions $\mathrm{supp}(\theta)\in\mathbf{B}(R)$ and $\frac{1}{n}H(\theta,\mathcal{D}_{n})\ge\epsilon$,
it follows that $\lambda_{n}\times\theta(F_{2})>\epsilon'$, where
$F_{2}$ is the set of all $(i,B_{1})\in\mathcal{N}_{n}\times\mathbf{B}$
such that
\[
\mathbb{P}_{i\le j\le i+k}\left\{ (\theta_{B_{1},i})_{B_{2},j}\text{ is not }(j,\epsilon')\text{-atomic with respect to }\Vert\cdot\Vert\right\} >\epsilon'.
\]
Since $\sigma$ is small with respect to $\epsilon'$, we get $\lambda_{n}\times\theta(F_{1}\cap F_{2})>\epsilon'/2$.
Moreover, using again the fact that $\sigma$ is small with respect
to $\epsilon'$, for every $(i,B_{1})\in F_{1}\cap F_{2}$ there exist
$i\le j\le i+k$ and $B_{2}\in\mathbf{B}$ so that $\xi:=(\theta_{B_{1},i})_{B_{2},j}$
is not $(j,\epsilon')$-atomic with respect to $\Vert\cdot\Vert$
and 
\[
\mu\left\{ x\::\:\xi.x\text{ is }(j,\sigma)\text{-atomic}\right\} >1-\sigma.
\]
Note that $\mathrm{supp}(\xi)\subset\mathrm{supp}(\theta)\subset\mathbf{B}(R)$
and $\mathrm{diam}(\mathrm{supp}(\xi))=O(2^{-j})$, and recall that
$\mathbf{B}(R)$ is compact. Thus, by Lemma \ref{lem:def =000026 prop of inv metric}
it follows that $\mathrm{diam}(\mathrm{supp}(\xi))=O_{R}(2^{-j})$
with respect to $\Vert\cdot\Vert$, which completes
the proof of the proposition.
\end{proof}

\subsection{Proof of Theorem \ref{thm:ent inc under conv}} \label{incproof}

We now rephrase the conclusion of Proposition \ref{prop:key prop for ent inc} in terms of the behaviour of certain mappings which we now introduce.

For $B_{1},B_{2}\in\mathrm{Mat}_{2,3}(\mathbb{R})$ let $p_{B_{1},B_{2}}\in\mathbb{R}[X_{1},X_{2}]$
be with
\[
p_{B_{1},B_{2}}(x):=\det\left(\begin{array}{cc}
\left\langle r_{B_{1},1},\tilde{x}\right\rangle  & \left\langle r_{B_{2},1},\tilde{x}\right\rangle \\
\left\langle r_{B_{1},2},\tilde{x}\right\rangle  & \left\langle r_{B_{2},2},\tilde{x}\right\rangle 
\end{array}\right)\text{ for }x\in\mathbb{R}^{2}.
\]
Note that $p_{B_1,B_2}(x)$ is closely related to the linear fractional transformations $\varphi_{B_1}$ and $\varphi_{B_2}$, for example if $x \in \R^2_{>0}$ and $B_1,B_2 \in \mathbf{B}$ then $p_{B_1,B_2}(x)=0$ is equivalent to $\varphi_{B_1}(x)=\varphi_{B_2}(x)$. Proposition \ref{prop:key prop for ent inc} will allow us to find $B_1$ and $B_2$ which are ``far'' from each other in a suitable sense, but for which we have control over the measure of the set of points $x$ for which $p_{B_1,B_2}(x)$ is close to zero.

\begin{lem}
\label{lem:exist B_1,B_2}For every $R>1$ there exists $C=C(R)>1$
so that for all $0<\epsilon<1$ and $0<\delta<\delta(\epsilon)$ the
following holds. Let $m\ge0$ and $\xi\in\mathcal{M}(\mathbf{B}(R))$
be with,
\begin{enumerate}
\item $\mathrm{diam}(\mathrm{supp}(\xi))\le R2^{-m}$ with respect to $\Vert\cdot\Vert$;
\item $\xi$ is not $(m,\epsilon)$-atomic with respect to $\Vert\cdot\Vert$;
\item $\mu\left\{ x\::\:\xi.x\text{ is }(m,\delta)\text{-atomic}\right\} >1-\delta$.
\end{enumerate}
Then there exist $B_{1}\in\mathrm{supp}(\xi)$ and $B_{2}\in\mathrm{Mat}_{2,3}(\mathbb{R})$
so that,
\begin{itemize}
\item $\Vert B_{2}\Vert\le C$ and $\Vert B_{2}-tB_{1}\Vert>C^{-1}\epsilon$
for all $t\in\mathbb{R}$;
\item $\mu\left\{ x\::\:\left|p_{B_{1},B_{2}}(x)\right|<C\delta\right\} >1-C\delta^{1/2}$.
\end{itemize}
\end{lem}

\begin{proof}
Let $R>1$ and $0<\epsilon<1$ be given, let $\delta>0$ be small
with respect to $\epsilon$, and let $m,\xi$ be as in the statement
of the lemma. By the third assumption and from the continuity of the
map taking $(B,x)\in\mathbf{B}\times\mathbb{R}_{>0}^{2}$ to $\varphi_{B}(x)$,
it follows easily that there exists a Borel map $L:\mathbb{R}_{>0}^{2}\rightarrow\mathbb{R}$
so that
\[
\mu\left\{ x\::\:\xi.x(\textup{B}(L(x),2^{1-m}\delta))>1-\delta\right\} >1-\delta.
\]
From this and by Fubini's theorem,
\[
\int\int1_{\textup{B}(L(x),2^{1-m}\delta)}(\varphi_{B}(x))d\mu(x)\:d\xi(B)>(1-\delta)^{2}>1-2\delta.
\]
Thus $\xi(Y)>1-(2\delta)^{1/2}$, where $Y$ is the set of all $B\in\mathrm{supp}(\xi)$
so that for
\[
E_{B}:=\left\{ x\in K_{\mathbf{A}}\::\:\varphi_{B}(x)\in \textup{B}(L(x),2^{1-m}\delta)\right\} 
\]
we have $\mu(E_{B})>1-(2\delta)^{1/2}$.

Fix some $B_{1}\in Y$. Since $B_{1}\in\mathrm{supp}(\xi)\subset\mathbf{B}(R)$,
we have $\Vert B_{1}\Vert=O_{R}(1)$. Write
\[
Q:=\{B\in\mathbf{B}(R)\::\:\Vert B-tB_{1}\Vert\le C^{-1}2^{-m}\epsilon\text{ for some }t\in\mathbb{R}\},
\]
where $C>1$ is a large constant depending only on $R$. Given $B\in Q$
there exists $t\in\mathbb{R}$ so that $\Vert B-tB_{1}\Vert\le C^{-1}2^{-m}\epsilon$.
By the definition of $\mathbf{B}$ and by assuming that $C$ is sufficiently
large, we must have $t>0$. Hence,
\[
|1-t|=\left||r_{B,2}|-|r_{tB_{1},2}|\right|\le|r_{B,2}-r_{tB_{1},2}|=O(C^{-1}2^{-m}\epsilon),
\]
which gives
\[
\Vert B-B_{1}\Vert\le\Vert B-tB_{1}\Vert+\Vert tB_{1}-B_{1}\Vert=O_{R}(C^{-1}2^{-m}\epsilon).
\]
From this, by assuming $C$ is sufficiently large, and by the second
assumption in the statement of the lemma, it follows that $\xi(\mathbf{B}\setminus Q)\ge\epsilon$.
Since $\xi(Y)>1-(2\delta)^{1/2}$ and $\delta$ is small with respect
to $\epsilon$, we get $\xi(Y\setminus Q)>\epsilon/2$. In particular,
there exists $B_{2}'\in\mathrm{supp}(\xi)\cap Y$ with $B_{2}'\notin Q$.
Set $B_{2}=2^{m}B_{2}'-(2^{m}-1)B_{1}$.

By the first assumption and since $B_{1},B_{2}'\in\mathrm{supp}(\xi)$,
\[
\Vert B_{2}\Vert\le\Vert B_{2}-B_{1}\Vert+\Vert B_{1}\Vert=2^{m}\Vert B_{2}'-B_{1}\Vert+O_{R}(1)=O_{R}(1).
\]
Since $B_{2}'\notin Q$ we have $\Vert B_{2}'-tB_{1}\Vert>C^{-1}2^{-m}\epsilon$
for all $t\in\mathbb{R}$, which implies $\Vert B_{2}-tB_{1}\Vert>C^{-1}\epsilon$
for all $t\in\mathbb{R}$.

Let $x\in E_{B_{1}}\cap E_{B_{2}'}$ be given. Let $f_{1},f_{2}\in\left(\mathrm{Mat}_{2,3}(\mathbb{R})\right)^{*}$
be with $f_{i}(B):=\left\langle r_{B,i},\tilde{x}\right\rangle $
for $i=1,2$ and $B\in\mathrm{Mat}_{2,3}(\mathbb{R})$. Since $x\in E_{B_{2}'}$,
\[
\left|\frac{f_{1}(B_{2}')}{f_{2}(B_{2}')}-L(x)\right|=\left|\varphi_{B_{2}'}(x)-L(x)\right|\le2^{1-m}\delta.
\]
Thus, from $B_{2}'\in\mathbf{B}(R)$ and $x\in K_{\mathbf{A}}$,
\[
\left|f_{1}(2^{m}B_{2}')-L(x)f_{2}(2^{m}B_{2}')\right|=O_{R}(\delta).
\]
Similarly, from $x\in E_{B_{1}}$,
\[
\left|f_{1}(cB_{1})-L(x)f_{2}(cB_{1})\right|=O_{R}(\delta)\text{ for }c=1,2^{m}-1.
\]
Thus,
\begin{multline*}
\left|f_{1}(B_{2})-L(x)f_{2}(B_{2})\right|\le\left|f_{1}(2^{m}B_{2}')-L(x)f_{2}(2^{m}B_{2}')\right|\\
+\left|f_{1}((2^{m}-1)B_{1})-L(x)f_{2}((2^{m}-1)B_{1})\right|=O_{R}(\delta).
\end{multline*}
From the last two formulas and since $|f_{j}(B_{i})|=O_{R}(1)$ for
$i,j=1,2$,
\begin{multline*}
\left|p_{B_{1},B_{2}}(x)\right|=\det\left(\begin{array}{cc}
f_{1}(B_{1}) & f_{1}(B_{2})\\
f_{2}(B_{1}) & f_{2}(B_{2})
\end{array}\right)\\
=\det\left(\begin{array}{cc}
f_{1}(B_{1})-L(x)f_{2}(B_{1}) & f_{1}(B_{2})-L(x)f_{2}(B_{2})\\
f_{2}(B_{1}) & f_{2}(B_{2})
\end{array}\right)=O_{R}(\delta).
\end{multline*}
Additionally, from $B_{1},B_{2}'\in Y$ it follows that $\mu(E_{B_{1}}\cap E_{B_{2}'})>1-(8\delta)^{1/2}$,
which completes the proof of the lemma.
\end{proof}

On the other hand, if $B_2 \notin B_1\R$ then Zariski density implies that $p_{B_1,B_2}$ cannot vanish $\mu$ almost everywhere.

\begin{lem}
\label{lem:no poly exists}Let $B_{1}\in\mathbf{B}$ and $B_{2}\in\mathrm{Mat}_{2,3}(\mathbb{R})$
be with $B_{2}\notin B_{1}\mathbb{R}$. Then $\mu\{x\;:\:p_{B_{1},B_{2}}(x)\ne0\}>0$.
\end{lem}

\begin{proof}
Let us first show that $p_{B_{1},B_{2}}\ne0$. Assume by contradiction
that this is not the case. Write $\mathrm{R}$ for the ring $\mathbb{R}[X_{1},X_{2}]$.
For $i,j=1,2$ let $q_{i,j}\in\mathrm{R}$ be with $q_{i,j}(x)=\left\langle r_{B_{i},j},\tilde{x}\right\rangle $
for $x\in\mathbb{R}^{2}$. From $p_{B_{1},B_{2}}=0$ it follows that
$q_{1,1}q_{2,2}=q_{1,2}q_{2,1}$. Since $B_{1}\in\mathbf{B}$, the
vectors $r_{B_{1},1},r_{B_{1},2}$ are linearly independent. From
this and $\deg(q_{1,j})\le1$ for $j=1,2$, it follows that at least
one of the polynomials $q_{1,1},q_{1,2}$ is an irreducible element
of $\mathrm{R}$. Without loss of generality assume that $q_{1,1}$
is irreducible. Since $\mathrm{R}$ is a unique factorization domain,
it follows that $q_{1,1}$ is a prime element of $\mathrm{R}$. Moreover,
$q_{1,1}$ does not divide $q_{1,2}$ since $r_{B_{1},1},r_{B_{1},2}$
are linearly independent. From these facts and $q_{1,1}q_{2,2}=q_{1,2}q_{2,1}$,
it follows that $q_{2,1}=cq_{1,1}$ for some $c\in\mathbb{R}$. Together
with $q_{1,1}q_{2,2}=q_{1,2}q_{2,1}$ this gives $q_{2,2}=cq_{1,2}$.
From $q_{2,1}=cq_{1,1}$ and $q_{2,2}=cq_{1,2}$ we obtain $B_{2}=cB_{1}$,
which contradicts $B_{2}\notin B_{1}\mathbb{R}$. Hence we must have
$p_{B_{1},B_{2}}\ne0$. In particular, $p_{B_{1},B_{2}}(y)\ne0$ for
some $(y_{1},y_{2})=y\in\mathbb{R}^{2}$.

Next, assume by contradiction that $p_{B_{1},B_{2}}(x)=0$ for $\mu$-a.e.
$x$. Given $u\in\I^{*}$ we have $\varphi_{A_{u}}\mu\ll\mu$,
and so $p_{B_{1},B_{2}}(\varphi_{A_{u}}(x))=0$ for $\mu$-a.e. $x$.
Thus, since $\I^{*}$ is countable and $\mathbf{S}=\{A_{u}\}_{u\in\I^{*}}$,
there exists $(z_{1},z_{2})=z\in\mathbb{R}_{>0}^{2}$ so that $p_{B_{1},B_{2}}(\varphi_{A}(z))=0$
for all $A\in\mathbf{S}$. Let $q:\mathrm{Mat}_{3,3}(\mathbb{R})\rightarrow\mathbb{R}$
be the polynomial map with
\[
q(A):=\det\left(\left\{ \left\langle r_{B_{i},j},\left(\left\langle r_{A,1},\tilde{z}\right\rangle ,\left\langle r_{A,2},\tilde{z}\right\rangle ,\left\langle r_{A,3},\tilde{z}\right\rangle \right)\right\rangle \right\} _{i,j=1,2}\right)\text{ for }A\in\mathrm{Mat}_{3,3}(\mathbb{R}).
\]
Direct computation shows that
\begin{equation}
q(A)=\left\langle r_{A,3},\tilde{z}\right\rangle ^{2}p_{B_{1},B_{2}}(\varphi_{A}(z))\text{ for }A\in\mathrm{SL}(3,\mathbb{R})\text{ with }\left\langle r_{A,3},\tilde{z}\right\rangle \ne0,\label{eq:q(A)=00003D for A with...}
\end{equation}
and so $q(A)=0$ for all $A\in\mathbf{S}$. Since $\mathbf{S}$ is
Zariski in $\mathrm{SL}(3,\mathbb{R})$, it follows that $q(A)=0$
for all $A\in\mathrm{SL}(3,\mathbb{R})$.

Set
\[
M:=\left(\begin{array}{ccc}
1 & 0 & y_{1}-z_{1}\\
0 & 1 & y_{2}-z_{2}\\
0 & 0 & 1
\end{array}\right).
\]
Since $M\in\mathrm{SL}(3,\mathbb{R})$ we have $q(M)=0$. On the other
hand, from (\ref{eq:q(A)=00003D for A with...}) and since
\[
\left\langle r_{M,3},\tilde{z}\right\rangle ^{2}=1\text{ and }p_{B_{1},B_{2}}(\varphi_{M}(z))=p_{B_{1},B_{2}}(y)\ne0,
\]
it follows that $q(M)\ne0$. This contradiction completes the proof
of the lemma.
\end{proof}
We are now ready to prove Theorem \ref{thm:ent inc under conv} by contradiction. We will assume that there is a sequence $\delta_k \to 0$ and $\theta_k$ satisfying the hypothesis of Theorem \ref{thm:ent inc under conv} for all $k \in \N$. Applying Proposition \ref{prop:key prop for ent inc} and Lemma \ref{lem:exist B_1,B_2} along this sequence will give a contradiction with Lemma \ref{lem:no poly exists}.
\begin{proof}[Proof of Theorem \ref{thm:ent inc under conv}]
Let $\epsilon':=\epsilon'(\epsilon,R)>0$ be as obtained in Proposition
\ref{prop:key prop for ent inc}. Additionally, for each $k\ge1$
let $\delta_{k}:=\delta(\epsilon,R,1/k)>0$ and $N_{k}:=N(\epsilon,R,1/k)\ge1$
be as obtained in Proposition \ref{prop:key prop for ent inc}. We
may clearly assume that $\epsilon'<1$.

Assume by contradiction that the theorem is false. Then for each $k\ge1$
there exist $n_{k}\ge N_{k}$ and $\theta_{k}\in\mathcal{M}(\mathbf{B}(R))$
so that $\frac{1}{n_{k}}H(\theta_{k},\mathcal{D}_{n_{k}})\ge\epsilon$
and$\frac{1}{n_{k}}H(\theta_{k}\ldotp\mu,\mathcal{D}_{n_{k}})\le\Delta+\delta_{k}$.
From this and by Proposition \ref{prop:key prop for ent inc}, it
follows that for each $k\ge1$ there exist $\xi_{k}\in\mathcal{M}(\mathbf{B}(R))$
and $m_{k}\ge0$ such that,
\begin{enumerate}
\item $\mathrm{diam}(\mathrm{supp}(\xi_{k}))=O_{R}(2^{-m_{k}})$ with respect
to $\Vert\cdot\Vert$;
\item $\xi_{k}$ is not $(m_{k},\epsilon')$-atomic with respect to $\Vert\cdot\Vert$;
\item $\mu\left\{ x\::\:\xi_{k}.x\text{ is }(m_{k},1/k)\text{-atomic}\right\} >1-1/k$.
\end{enumerate}
Let $C>1$ be large in a manner depending only on $R$. From Lemma
\ref{lem:exist B_1,B_2} it follows that for all $k\ge1$ large enough
there exist $B_{k,1}\in\mathrm{supp}(\xi_{k})\subset\mathbf{B}(R)$
and $B_{k,2}\in\mathrm{Mat}_{2,3}(\mathbb{R})$ so that,
\begin{enumerate}
\item $\Vert B_{k,2}\Vert\le C$ and $\Vert B_{k,2}-tB_{k,1}\Vert>C^{-1}\epsilon'$
for all $t\in\mathbb{R}$;
\item $\mu(E_{k})<C/k^{1/2}$, where $E_{k}:=\left\{ x\in\mathbb{R}^{2}\::\:\left|p_{B_{k,1},B_{k,2}}(x)\right|\ge C/k\right\} $.
\end{enumerate}
Recall that $\mathbf{B}(R)$ is compact with respect to $d_{\mathbf{B}}$,
which implies that it is compact also with respect to $\Vert\cdot\Vert$
(see Lemma \ref{lem:def =000026 prop of inv metric}). Hence, by moving
to a subsequence without changing notation, there exist $B_{1}\in\mathbf{B}(R)$
and $B_{2}\in\mathrm{Mat}_{2,3}(\mathbb{R})$ so that $\Vert B_{i}-B_{k,i}\Vert\overset{k}{\rightarrow}0$
for $i=1,2$. It is clear that $B_{2}\notin B_{1}\mathbb{R}$. By
moving to a subsequence without changing notation, we may also assume
that $\sum_{k\ge1}\mu(E_{k})<\infty$. Thus, by the Borel-Cantelli
lemma, we have $\mu(E)=0$ where $E:=\cap_{n\ge1}\cup_{k\ge n}E_{k}$.
For each $x\in\mathbb{R}^{2}\setminus E$ it clearly holds that $p_{B_{k,1},B_{k,2}}(x)\overset{k}{\rightarrow}0$.
Hence, since $p_{B_{k,1},B_{k,2}}\overset{k}{\rightarrow}p_{B_{1},B_{2}}$
pointwise, it follows that $p_{B_{1},B_{2}}(x)=0$ for $\mu$-a.e.
$x$. All of this contradicts Lemma \ref{lem:no poly exists}, which
completes the proof of the theorem.
\end{proof}

We will now show how Theorem \ref{thm:ent inc under conv2} follows from Theorem \ref{thm:ent inc under conv} by the invariance properties of the metric $d_{\mathbf{B}}$ defining our partition. 

\begin{proof}[Proof of Theorem \ref{thm:ent inc under conv2}] Fix $\epsilon, R>0$ and let $BA \in \mathrm{supp} \,\theta$ satisfy Theorem \ref{thm:ent inc under conv2}(a). Let $Y= \begin{pmatrix} \frac{1}{c_{A,B}} & -\frac{t_{A,B}}{c_{A,B}}\\0&1 \end{pmatrix}$. Then $Y(\theta)=\theta'$ where $\theta'$ is a measure supported on $\mathbf{B}(R)$ by Theorem \ref{thm:ent inc under conv2}(a) and Lemma \ref{lem:def =000026 prop of inv metric}(1). By Theorem \ref{thm:ent inc under conv2}(b) and \eqref{affinetingB}, $\frac{1}{n} H(\theta', \mathcal{D}_n) \geq \frac{\epsilon}{2}$ for sufficiently large $n$. Now, applying Theorem  \ref{thm:ent inc under conv} we deduce that $\frac{1}{n} H(\theta' . \mu, \mathcal{D}_n) \geq \Delta+\delta$ for $n \geq N(\epsilon, \delta, R)$. By \eqref{affineting}
\begin{align*}
H(\theta'.\mu, \mathcal{D}_n)=H(Y(\theta) . \mu, \mathcal{D}_n)&=H(\theta.\mu, \mathcal{D}_{n+\log c_{A,B}})+O(1).
\end{align*}
This proves the result.
\end{proof}

We also will require the following simpler lower bound.

\begin{prop} \label{nodrop} 
Let $R>0$, $B_0 \in \mathbf{B}_{\mathrm{o}}$ and suppose that for some $A_0 \in \mathbf{S}$,  $\mathrm{supp}\, \theta \subset \textup{B}(B_0A_0,R)$. Then
$$H(\theta.\mu, \D_n) \geq \inf_{V \in \mathbf{Y}} H(\varphi_{B_V}\mu, \D_{n-\log c_{A_0,B_0}})+O_R(1).$$
\end{prop}

\begin{proof}
First fix $B=\begin{pmatrix}r_{B,1}\\ r_{B,2}\end{pmatrix} \in \mathbf{B}(R)$.
%by triangle inequality enough to compare B to M from recsale lemma. this is distance between r_{B,1} and v/\norm{v}. but \norm{v} \leq \norm{r_{B,1}} \leq this distance plus 1.
 Define $B'=\begin{pmatrix}\frac{v}{\norm{v}}\\  r_{B,2}\end{pmatrix}\in \mathbf{B}_{\mathrm{o}}$ where $v=r_{B,1}-\langle r_{B,1},r_{B,2}\rangle r_{B,2}$, noting that $\norm{v}=O_R(1)$ since $B \in \mathbf{B}(R)$. Similarly to the proof of Lemma \ref{rescale}, we see that $\varphi_B=\norm{v}\varphi_{B'}+\langle r_{B,1}, r_{B,2}\rangle$. By \eqref{affineting}
$$
H(\varphi_B\mu, \D_n)=H(\varphi_{B'}\mu, \D_n)+O(\norm{v}) \geq \inf_{V \in \mathbf{Y}} H(\varphi_{B_V}\mu,\D_n)+O_R(1).
$$
Hence if $\mathrm{supp}\,\theta \subset \mathbf{B}(R)$ then by concavity of entropy
\begin{equation}
H(\theta.\mu, \D_n)= H\left(\left(\int \delta_B d\theta(B)\right). \mu, \D_n\right) \geq\int H(\varphi_B \mu, \D_n) d\theta(B)\geq  \inf_{V \in \mathbf{Y}}H(\varphi_{B_V}\mu,\D_n)+O_R(1).\label{nodrop1} \end{equation}
Now suppose for some $A_0 \in \Se$ and $B_0 \in \mathbf{B}_{\mathrm{o}}$, $\mathrm{supp}\,\theta \subset \mathbf{B}(B_0A_0,R)$. Let $Y(\theta)=\theta'$ where  $Y= \begin{pmatrix} \frac{1}{c_{A_0,B_0}} & -\frac{t_{A_0,B_0}}{c_{A_0,B_0}}\\0&1 \end{pmatrix}$ where $c_{A_0,B_0}$ and $t_{A_0,B_0}$ is given by Lemma \ref{rescale}. By Lemma \ref{lem:def =000026 prop of inv metric}(1) and Lemma \ref{rescale}, $\theta'$ has support in $\mathbf{B}(R)$. Thus by \eqref{nodrop1},
$$H(\theta'.\mu, \D_n) \geq \inf_{V \in \mathbf{Y}} H(\varphi_{B_V}\mu, \D_n)+O_R(1).$$
Moreover
\begin{align*}
H(\theta'.\mu, \D_n)=H(Y( \theta).\mu, \D_n)&= H(\theta.\mu, \D_{n+\log c_{A_0,B_0}})+O(1).
\end{align*}
The last two displayed equations give the result.
\end{proof}

\section{Proof of Theorem \ref{thm:ESC --> Delta =00003D correct val} } \label{finalproof}

In this section we complete the proof of Theorem \ref{thm:ESC --> Delta =00003D correct val}, and thus Theorem \ref{msrthm}. The proof will be by contradiction. Recall the definition of $\Xi_n^B$ from \eqref{xi} and define $\theta_n^B= \sum_{\i \in \Xi_{n}^B} p_{\i} \delta_{BA_\i}$. We begin with two results which will be combined to show that if the Diophantine property holds and $\Delta< \frac{H(p)}{\chi_1(p)-\chi_2(p)}$ then a non-trivial proportion of the components of $\theta_n^B$ have non-negligible entropy. The entropy increase result Theorem \ref{thm:ent inc under conv2} will then be applied to these components to derive the desired contradiction.
%Before completing the proof of Theorem \ref{thm:follows from LY formula} we must prove the following entropy convergence result. 

\begin{prop}\label{dimapprox}
For $\nu$ almost every $V \in \bf{Y}$, 
$$\frac{1}{n} H(\theta_n^{B_V}, \mathcal{D}_0^{\mathbf{B}}) \to \Delta.$$
\end{prop}

\begin{proof}
By definition of $\Delta$, $\frac{1}{n} H(\varphi_{B_V} \mu, \mathcal{D}_n) \to \Delta$ for $\nu$ almost every $V \in \bf{Y}$. We now fix such a typical $V$. Let $w=(0,0)$. Given any $\i \in\theta_n^{B_V}$, $\varphi_{BA_\i} \mu$ is contained in a set of diameter $|K_\A|2^{-n}$ therefore 
$$\frac{1}{n} H(\theta_n^{B_V} . w, \mathcal{D}_{n}) \to \Delta.$$
This directly implies that
\begin{equation}\frac{1}{n} H(\theta_n^{B_V} , \rho_w^{-1}(\mathcal{D}_{n})) \to \Delta \label{entropyconv} \end{equation}
where $\rho_w: \mathbf{B} \to \R^2$ is defined as $\rho_w(B)=\varphi_B(w)$. We think of $\rho_w^{-1}(\D_n)$ as a partition and refer to any element $\rho_w^{-1}(D)$ ($D \in \D_n$) as an atom. To prove the proposition it suffices to show that \eqref{entropyconv} implies that
\begin{equation} \label{entropyconv2}
\frac{1}{n} H(\theta_n^{B_V} , \mathcal{D}_{0}^{\mathbf{B}}) \to \Delta .
\end{equation}
By \eqref{commens}, in order to prove \eqref{entropyconv2} it is sufficient to show that
\begin{enumerate}
\item[(a)] every atom of $\D_0^{\mathbf{B}}$ that intersects $\mathrm{supp}\, \theta_n^{B_V}$ can be covered by $O(1)$ atoms of $\rho_w^{-1}(\D_n)$,
\item[(b)] every atom of $\rho_w^{-1}(\mathcal{D}_{n}) $ that intersects $\mathrm{supp}\, \theta_n^{B_V}$ can be covered by $O(1)$ atoms of $\D_0^{\mathbf{B}}$.
\end{enumerate}
To prove (a), notice that if $E \in \mathcal{D}_0^{\mathbf{B}}$ is an atom and $B_VA, B_VA' \in E$ then $d_{\mathbf{B}}(M, D^{-1}B_VA') =O(1)$ where $D=\begin{pmatrix} c_{A,B_V}&t_{A,B_V} \\0&1 \end{pmatrix}$ and $M=M_{A,B_V} \in \mathbf{B}_{\mathrm{o}}$ is given by applying Lemma \ref{rescale} to $\varphi_{B_VA}$. Hence $D^{-1}B_VA' \in \mathbf{B}(R)$ for some $R=O(1)$. In particular since $\mathbf{B}(R)$ is compact, $\norm{M- D^{-1}B_VA'}=O(1)$  by Lemma \ref{lem:def =000026 prop of inv metric}(3). This implies $|\varphi_M(w)-\varphi_{D^{-1}B_VA'}(w)| =O(1)$ which in turn implies $|\varphi_{B_VA}(w)-\varphi_{B_VA'}(w)|=O(c_{A,B_V})=O(2^{-n})$, by definition of $\Xi_n^{B_V}$. This proves (a).

Next we prove (b). Let $E \in \rho_w^{-1}(\mathcal{D}_{n})$ be an atom for which $\mathrm{supp}\, \theta_n^{B_V} \cap E \neq \emptyset$. Write $X_n=\begin{pmatrix}2^n&0\\0&1\end{pmatrix}$. Note that there exists some atom $[a,a+1]$ of $\mathcal{D}_0$, depending only on $E$, such that for any $B_VA \in\mathrm{supp}\, \theta_n^{B_V} \cap E$, 
\begin{equation}
\varphi_{X_nB_VA}(w) \in [a,a+1].
\label{a}
\end{equation}  
By applying Lemma \ref{rescale} to $\varphi_{B_VA}$ we can express $B_VA=\begin{pmatrix}c_{A,B_V} &t_{A,B_V} \\0&1 \end{pmatrix} M_{A,B_V}$ for $M_{A,B_V} \in \mathbf{B}_{\mathrm{o}}$. There exists a compact (with respect to $d_{\mathbf{B}}$) set $K$, which is independent of $n$, such that for any $B_VA \in \mathrm{supp}\, \theta_n^{B_V} $,  $Y=\begin{pmatrix}2^nc_{A,B_V}&0\\0&1\end{pmatrix} M_{A,B_V} \in K$. Moreover by Lemma \ref{lem:def =000026 prop of inv metric}(3), there exists some compact subset of $\R$, say $[-b,b]$ such that for any $Y \in K$, 
\begin{equation}
\varphi_Y(w) \in [-b,b].
\label{b}
\end{equation}
By \eqref{a} and \eqref{b} and since $X_nB_VA=\begin{pmatrix}1&2^nt_{A,B_V} \\0&1 \end{pmatrix}\begin{pmatrix}2^nc_{A,B_V}&0\\0&1\end{pmatrix} M_{A,B_V}$, there exists a compact interval $S \subset \R$, whose diameter can be bounded above in terms of $b$, such that 
$$\mathrm{supp} \,\theta_n^{B_V} \cap E \subset \bigcup_{s \in S} \bigcup_{Y \in K}\begin{pmatrix} 1&s\\0&1 \end{pmatrix} Y.$$
 Notice that by the triangle inequality the $d_{\mathbf{B}}$-diameter of the right hand side is uniformly bounded above: if $s,s' \in S$ and $Y, Y' \in K$ then 
\begin{align*}
d_{\mathbf{B}}\left(\begin{pmatrix} 1&s\\0&1 \end{pmatrix}Y, \begin{pmatrix} 1&s'\\0&1 \end{pmatrix}Y'\right) &\leq d_{\mathbf{B}}\left(\begin{pmatrix} 1&s\\0&1 \end{pmatrix}Y, \begin{pmatrix} 1&s'\\0&1 \end{pmatrix}Y\right)+ d_{\mathbf{B}}\left(\begin{pmatrix} 1&s'\\0&1 \end{pmatrix}Y, \begin{pmatrix} 1&s'\\0&1 \end{pmatrix}Y'\right).
\end{align*}
By Lemma \ref{lem:def =000026 prop of inv metric}(3) the first term is bounded above by $O(\mathrm{diam}(S))$ and by \ref{lem:def =000026 prop of inv metric}(1) the second term is uniformly bounded above by the $d_{\mathbf{B}}$- diameter of $K$. In particular, the $d_{\mathbf{B}}$-diameter of $E$ is uniformly bounded above, proving (b).

\end{proof}

Next we show that the Diophantine property (Definition \ref{ESC}) implies that almost every projected system is ``exponentially separated''.

\begin{lem} \label{projectedESC}
Suppose $\mathbf{A}$ satisfies the Diophantine property. Then there exists $C_0>1$, $\mathbf{Y}' \subset \mathbf{Y}$ with $\nu(\mathbf{Y}')=1$ such that the following holds. For all $V \in \mathbf{Y}'$ there exists $N=N(V)$ such that for all $\i,\j \in \Xi_n^{B_V}$ with $\i \neq \j$,
\begin{equation}\D_{C'n}(B_VA_\i) \neq \D_{C'n}(B_VA_\j) \label{diffcubes} \end{equation}
for all $ n \geq N$ and $C \geq C_0$. 
\end{lem}

\begin{proof}
We begin by showing that if $\mathbf{A}$ satisfies the Diophantine property then there exists $b>0$ such that for all $\i,\j \in \I^*$ with $|\i| \geq |\j|$ where $\j$ is not a prefix of $\i$, 
\begin{equation}
\norm{A_\i-A_\j} \geq b^{|\i|}. \label{prefix}
\end{equation}
Suppose this is false. Let $\kappa$ be sufficiently large that for all $\i \in \I^m$, $\norm{A_\j} \leq \kappa^m$. Put $b=\frac{\epsilon^2}{2\kappa}$, where $\epsilon$ is given by Definition \ref{ESC}. Then there exists $\i, \j \in \I^*$ with $n=|\i|\geq |\j|=m$ such that $\j$ is not a prefix of $\i$ with $\norm{A_\i-A_\j}<b^n$. Note that necessarily $\i\j \neq \j\i$. Now,
$$\epsilon^{2n}\leq \epsilon^{n+m} \leq \norm{A_{\i\j}-A_{\j\i}} \leq \norm{A_{\i\j}-A_{\j\j}}  +\norm{A_{\j\j}-A_{\j\i}} \leq 2\kappa^nb^n<\epsilon^{2n}$$
yielding a contradiction and proving \eqref{prefix}.

Recall that we have equipped $\mathrm{Gr}_2(3)$ with the metric induced from $d_{\mathbb{P}}$, i.e. the distance between two planes is given by the angle between their normal vectors. By \cite[Chapter V1.4]{BL}, $\hd \nu >0.$ We will construct a set $E\subset \bf{Y}$ of dimension $\frac{\hd\nu}{2}$ such that for all $V \in \mathbf{Y}\setminus E$, \eqref{diffcubes} holds for the collection $\{B_VA_\i\}_{\i \in \Xi_n^{B_V}}$. This will be sufficient to prove the lemma.

For each $n \geq1 $ define
$$U_n=\{(\i,\j) \in \I^*\times \I^*: \textnormal{$|\i|,|\j| \leq n$, $\i$ is not a prefix of $\j$ and $\j$ is not a prefix of $\i$}\},$$
noting that $\# U_n \leq  (|\I|)^{2(n+1)}$. From Lemma \ref{scaling} and positivity of $\A$ it follows that there exists $\ell \in \mathbb{N}$ such that for all $\i\in \Xi_n^{B_V}$, $|\i| \leq n\ell$ hence for all $\i,\j \in\Xi_n^{B_V}$  with $\i \neq \j$, $(\i,\j) \in U_{n\ell}$. Let $b$ be given by \eqref{prefix} so that $\norm{A_\i-A_\j} \geq b^n$ for all $(\i,\j) \in U_n$. For each $(\i,\j) \in U_n$ we can choose $w_{\i,\j}^n \in \R^3$ with $\norm{w_{\i,\j}^n}=1$ such that
$$\norm{A_\i w_{\i,\j}^n-A_\j w_{\i,\j}^n} \geq b^n.$$
 Let $s=\frac{\hd \nu}{2}$ and $0<c<\frac{b}{|\I|^{2/s}}$ and set
$$E_{\i,\j}^n=\{V \in \mathbf{Y}:\norm{B_VA_\i w_{\i,\j}^n-B_VA_\j w_{\i,\j}^n}<c^n\}.$$
It follows from the sine rule that there exists $\lambda>0$ such that for all $x,y \in \R^3$ and $\delta>0$, the set of $V \in \bf{Y}$ satisfying $\norm{B_Vx-B_Vy}<\delta$ has diameter at most $\frac{\lambda\delta}{\norm{x-y}}$.

Applying this estimate on the diameter with $\delta=c^n$, $x=B_VA_\i w_{\i,\j}^n$ and $y=B_VA_\j w_{\i,\j}^n$ yields that $\mathrm{diam} E_{\i,\j}^n \leq \lambda(c/b)^n$. Now, define
$$E=\bigcap_{N=1}^\infty \bigcup_{n=N}^\infty \bigcup_{(\i,\j) \in U_n} E_{\i,\j}^n.$$
Note that for all $V \in \mathbf{Y} \setminus E$, there exists $N(V)$ such that for $n \geq N(V)$, $\norm{B_VA_\i-B_VA_\j} \geq c^n$ for all $(\i,\j) \in U_n$. This implies that $\{B_VA_\i\}_{\i \in \Xi_n^{B_V}}$ are $c^{n\ell}$ separated in $\mathbf{B}$ with respect to $\norm{\cdot}$. Now \eqref{diffcubes} follows from Lemma \ref{lem:def =000026 prop of inv metric} and \S \ref{dyadiclike}(2).
%, and thus $\frac{1}{C}c^{n\ell}$ separated in $\mathbf{B}$ with respect to $d_{\mathbf{B}}$, where $C$ is given by Lemma \ref{lem:def =000026 prop of inv metric}.  By \S \ref{dyadiclike}(2) we can choose $C_0=C_0(s)$ sufficiently large such that for all $C' \geq C_0(s)$, \eqref{diffcubes} holds for $V \in \mathbf{Y} \setminus E(s)$.

It remains to show that $\nu(\mathbf{Y} \setminus E)=1$. Indeed, letting $\mathcal{H}^s$ denote the $s$-dimensional Hausdorff measure on $\mathrm{Gr}_2(3)$,
$$\mathcal{H}^s(E) \leq \lim_{N \to \infty} \sum_{n=N}^\infty \sum_{(\i,\j) \in U_n} (\mathrm{diam} E_{\i,\j}^n)^s \leq \lim_{N \to \infty}(|\I|)^{2(n+1)} \frac{c^{sn}}{b^{sn}}=0$$
which proves that $\hd E \leq s$. In particular $\nu(\mathbf{Y} \setminus E)=1$.
\end{proof}

\begin{proof}[Proof of Theorem \ref{thm:ESC --> Delta =00003D correct val}] Assume for a contradiction that $\Delta< \frac{H(p)}{\chi_1(p)-\chi_2(p)}$. Let $V \in \mathbf{Y}'$ belong to the set of full measure for which Theorem \ref{thm:follows from LY formula} holds, where $\mathbf{Y}'$ is given by Lemma \ref{projectedESC}.  Let $C \geq C_0$, where $C_0$ is also given by Lemma \ref{projectedESC}. Then by \eqref{compentropy}
\begin{align*}
\mathbb{E}_{i=0} \left(\frac{1}{n} H((\theta_n^{B_V})_{B,i},\D_{Cn})\right)&=\frac{1}{n} H(\theta_n^{B_V},\D_{Cn}|\D_0) \\
&=\frac{1}{n} H(\theta_n^{B_V},\D_{Cn})-\frac{1}{n} H(\theta_n^{B_V},\D_0).
\end{align*}
We claim that 
\begin{equation}\label{smb}
\frac{1}{n} H(\theta_n^{B_V}, \D_{Cn}) \to \frac{H(p)}{\chi_1(p)-\chi_2(p)}.
\end{equation}
To see this, recall the definition of $J(n,B_V)$ from \S \ref{randomcyl} and note that by Lemma \ref{projectedESC},
\begin{align*}
\frac{1}{n} H(\theta_n^{B_V}, \D_{Cn})&= -\mathbb{E}\left(\frac{1}{n} \log p_{J(n,B_V)}\right) \\
&= -\mathbb{E}\left(\frac{|J(n,B_V)|}{n}\frac{1}{|J(n,B_V)|} \log p_{J(n,B_V)}\right).
\end{align*}
By the Shannon-McMillan-Breiman theorem $\frac{1}{|J(n,B_V)|} \log p_{J(n,B_V)} \to H(p)$. Moreover by Proposition \ref{lyapdiff}, $\frac{|J(n,B_V)|}{n} \to \frac{1}{\chi_1(p)-\chi_2(p)}$. \eqref{smb} follows by the dominated convergence theorem.

Now, assume for a contradiction that $H(p)>(\chi_1(p)-\chi_2(p))\Delta$. By \eqref{smb}, Proposition \ref{dimapprox} and Lemma \ref{projectedESC} there exists $\epsilon'>0$ such that
$$\mathbb{E}_{i=0} \left(\frac{1}{n} H((\theta_n^{B_V})_{B,i},\D_{Cn})\right) \geq \epsilon'$$
for all $n$ sufficiently large, in particular for some $\epsilon>0$,
\begin{align}
\mathbb{P}_{i=0} \left(\frac{1}{Cn} H((\theta_n^{B_V})_{B,i},\D_{Cn})>\epsilon\right)>\epsilon \label{finalproofeq}
\end{align}
for all $n$ sufficiently large. 

By Lemmas \ref{rescale} and \ref{scaling} and definition of $J(n,B_V)$, there exists a constant $C'$ such that
\begin{equation}-C' \leq n + \log c_{A_{J(n,B_V)},B_V} \leq C' \label{kappa} \end{equation}
for all $n \in \N$.

%Note that we can easily obtain the following analogue of Lemma \ref{measure-bound} for $J(n,B_V)$ rather than $I(n)$: for all $\kappa>0$, for all $0<\theta<\theta(\kappa)$ if $n \geq N(\kappa,\theta)$ then for all $V \in \mathbf{Y}$,
%$$\mathbb{P}\left(d(L_2(A_{J(n,B_V)}),V)<\frac{\pi}{2}-\theta\right)>1-\kappa.$$ In particular by Lemma \ref{scaling}, there exists $C(\kappa)>0$ such that \begin{equation} \mathbb{P}\left(-C(\kappa) \leq n + \log c_{A_{J(n,B_V)},B_V} \leq C(\kappa)\right)>1-\kappa\label{kappa}\end{equation} for all $n \geq N(\kappa)$. 

Let $R$ be given by \S \ref{dyadiclike}(2) and let $\delta=\delta(\epsilon, R)$ be given by Theorem \ref{thm:ent inc under conv2}. Note that for any non-trivial level 0 component $\theta$ of $\theta_n^{B_V}$ there exists $\i \in \Xi_n^{B_V}$ such that $B_VA_\i \in \mathrm{supp}\, \theta$ and $\mathrm{supp} \,\theta \subset \textup{B}(B_VA_\i,2R)$.  For $n \geq N(\epsilon,R,\delta,\kappa)$ 
\begin{align*}
\Delta+o(1)= \frac{1}{Cn} H(\varphi_{B_V}\mu,\D_{Cn}) &\geq \mathbb{E}_{i=0}\left(\frac{1}{Cn} H((\theta_n^{B_V})_{B,i} . \mu, \D_{Cn})\right) \\
&\geq \frac{C+1}{C} \mathbb{E}_{i=0}\left(\frac{1}{(C+1)n} H((\theta_n^{B_V})_{B,i} . \mu, \D_{(C+1)n-n})\right) \\
&\geq \frac{C+1}{C} \epsilon(\Delta+\delta) +(1-\epsilon)\Delta+o(1) \\
&=\frac{C+1}{C} \Delta+\epsilon\delta+o(1)
\end{align*}
where the first line follows by the definition of dimension and concavity of entropy and the third line by Theorem \ref{thm:ent inc under conv2}, the definition of $\Xi_n^{B_V}$, Proposition \ref{nodrop}, \eqref{finalproofeq} and \eqref{kappa}. By choosing $C$ sufficiently large and $\kappa$ sufficiently small it follows that there exists $\epsilon''>0$ such that for all sufficiently large $n$, $\frac{1}{Cn} H(\varphi_{B_V}\mu,D_{Cn}) >\Delta+\epsilon''$ which is a contradiction. \end{proof}

\section{Proofs of Theorems \ref{rauzythm} and \ref{setthm}} \label{finalsection}

In this section we use Theorem \ref{thm:follows from LY formula} to prove the dimension results for sets, Theorems \ref{setthm} and \ref{rauzythm}.

\subsection{Proof of Theorem \ref{setthm}}

We begin by considering positive systems $\A \subset \SL_{>0}$. To obtain a lower bound we will apply the recent work of Morris and Sert \cite{morris2023variational} to find stationary measures whose Hausdorff dimension approximates the affinity dimension $s_\A$. The upper bound  in  Theorem \ref{setthm} will be a natural covering argument.
\subsubsection{Proof of lower bound}

Throughout this section we fix a finite set $\A=\{A_i\}_{i \in \I} \subset \SL_{>0}$ which generates a semigroup $\Se$ which is Zariski dense in $\SL$ and satisfies the SOSC. The lower bound will combine the variational principle of Cao, Feng and Huang \cite{cao2008thermodynamic} with recent results of Morris and Sert \cite{morris2023variational} to construct measures whose Hausdorff dimension approximate the affinity dimension. 

Denote $\Sigma=\I^\N$. Let $\mathcal{M}_\sigma(\Sigma)$ denote the $\sigma$- invariant measures. For $m\in \mathcal{M}_\sigma(\Sigma)$ define 
\begin{align*}
\Lambda_s(m)= \lim_{n \to \infty} \frac{1}{n} \int \log \phi^s(A_{\i|n}) dm(\i) =
\begin{cases}
-s(\chi_1(m)-\chi_2(m)) &\mbox{if } s \in [0,1]\\
 -(\chi_1(m)-\chi_2(m))-(s-1)(\chi_1(m)-\chi_3(m)) &\mbox{if } s \in [1,2]\\
-s(2\chi_1(m)-\chi_2(m)-\chi_3(m)) &\mbox{if } s \geq 2
\end{cases}
\end{align*}

Let $P_\A(s)$ denote the pressure \eqref{pressure}. Since $\A \subset \SL_{>0}$ then  $\phi^s$ is almost-submultiplicative (Proposition \ref{submult}) and thus we obtain the following special case of \cite[Theorem 1.1]{cao2008thermodynamic}. 

\begin{prop}\label{vp}For each $s \geq 0$,
$$P_\A(s)= \sup \{H(m)+\Lambda_s(m): m \in \mathcal{M}_\sigma(\Sigma)\}$$
where the supremum is attained by some ergodic measure. We refer to this measure as the equilibrium state.
\end{prop}

We will also require the following special case of \cite[Proposition 4.1]{morris2023variational}. 

\begin{prop}\label{morrissert}
Suppose $\A\subset \GL$ is irreducible. Let $G$ denote the Zariski closure of $\Se$ and $G_c$ denote the unique connected component of $G$ (with respect to the Zariski topology) which contains the identity. Let $m \in \mathcal{M}_\sigma(\Sigma)$. 

For all $\epsilon>0$ there exists $n \geq 1$ and $\J \subset \I^n$ such that:
\begin{enumerate}
\item the Zariski closure of $\{A_\j: \j \in \J^*\}$ is $G_c$,
\item $\# \J \geq e^{n(H(m)-\epsilon)}$,
\item for all $\ell \in \{1,2,3\}$,
$$\left|\sum_{r=1}^\ell \log \alpha_r(A_\j)-n|\j| \sum_{r=1}^\ell \chi_r(\mu) \right| \leq n|\j|\epsilon.$$
\end{enumerate}
\end{prop}

We are now ready to prove the lower bound in Theorem \ref{setthm}. We will prove it in the case that $s_\A \in (1,2]$, the other cases are similar. Let $\delta>0$ and choose $s>1$ such that $s_\A-\delta<s<s_\A$. Let $m \in \mathcal{M}_\sigma(\Sigma)$ be the equilibrium state associated to $P_\A(s)$. Since $P_\A(s)>0$ we can choose $\epsilon>0$ such that $P_\A(s) >7\epsilon$. 

We will now apply Proposition \ref{morrissert} with this choice of $m$ and $\epsilon$. Note that since $\SL$ is a Zariski closed and Zariski connected subset of $\GL$, $G=G_c=\SL$. Let $n \in \N$ and $\J \subset \I^n$ be guaranteed by that proposition and let $\beta$ be the uniform Bernoulli measure on $\J^\N$. Then Proposition \ref{morrissert} says that 
\begin{enumerate}
\item the Zariski closure of $\{A_\j: \j \in \J^*\}$ is $\SL$,
\item $H(\beta) \geq n(H(m)-\epsilon)$,
\item for all $\ell \in \{1,2,3\}$,
$$\left|\sum_{r=1}^\ell \chi_r(\beta)-n \sum_{r=1}^\ell \chi_r(m) \right| \leq n\epsilon.$$
\end{enumerate}

Since $\{A_j\}_{j \in \J}$  is a set of positive matrices which generate a Zariski dense semigroup in $\SL$ and which satisfy the SOSC, our Theorem \ref{msrthm} says that $\hd\Pi \beta=\ld \Pi\beta$. 

Moreover since
$$\Lambda_s(\cdot)=((\chi_1(\cdot)+\chi_2(\cdot))-2\chi_1(\cdot)+(s-1)((\chi_1(\cdot)+\chi_2(\cdot))-\chi_1(\cdot))$$
it follows from (3) that
$$|\Lambda_s(\beta)-n\Lambda_s(m)| \leq 5n\epsilon.$$
In particular
\begin{align*}
H(\beta)+\Lambda_s(\beta) &\geq nH(m)+n\Lambda_s(m)-6n\epsilon \\
&\geq n\epsilon(7-6)=\epsilon'>0.
\end{align*}
Rearranging this gives
$$\hd \Pi \beta=\ld \Pi\beta=1+ \frac{H(\beta)-(\chi_1(\beta)-\chi_2(\beta))}{\chi_1(\beta)-\chi_3(\beta)} \geq s+\frac{\epsilon'}{\chi_1(\beta)-\chi_3(\beta)} >s \geq s_\A-\delta.$$
Taking $\delta \to 0$ proves the result.

\subsubsection{Proof of upper bound}\label{sec:UB}

The upper bound will be a covering argument. For $A \in \SL_{>0}$ we will consider the the singular value decomposition of $A=VDU$ where $V,U$ are orthogonal matrices and
\begin{equation} \label{svd}
D=\begin{pmatrix}\alpha_2(A)&0&0\\0&\alpha_3(A)&0\\0&0&\alpha_1(A)\end{pmatrix}.
\end{equation} 

We begin with the following simple lemma.

\begin{lem} \label{singvector}
Let $\A\subset \SL_{>0}$ be finite or countable and suppose there exists a closed cone $K \subset \R^3_{>0}$ such that for all $A \in \A$, $A(\R^3_{\geq 0}) \subset K$ and $A^T(\R^3_{\geq 0}) \subset K$.

 Then there exists $C>0$ such that for all balls $\textup{B}(x,r) \subset \R^2$ and all $A \in \A$,
\begin{enumerate}[(a)]
\item $\varphi_U(\textup{B}(x,r))$ is contained in a ball of radius $Cr$,
\item $\varphi_V(\textup{B}(x,r))$ is contained in a ball of radius $Cr$,
\end{enumerate}
where $V,U$ come from the singular value decomposition \eqref{svd}. 
\end{lem}

\begin{proof}Let $\mathbf{u}_A \in \mathrm{P}(\R^3)$ be the eigenvector corresponding to the leading eigenvalue of $A^TA$ and let $\mathbf{v}_A=A\mathbf{u}_A$. Recall the definition of $F$ from \S \ref{induced}. Let $u_A=F^{-1}(\mathbf{u}_A)$ and $v_A=F^{-1}(\mathbf{v}_A)$ be the corresponding points in $\R^2$. By assumption $\mathbf{u}_A, \mathbf{v}_A \in \mathrm{P}(K)$, hence $u_A, v_A \in F^{-1}(\mathrm{P}(K))$ which is a compact subset of $\R^2$. Since $U$ is an orthogonal matrix which rotates $\mathbf{u}_A$ to the line in the direction $(0,0,1)$, followed by a (possibly trivial) rotation of the plane spanned by $(1,0,0)$ and $(0,1,0)$, there exists a constant $C_A=O(|u_A|)$ such that the diameter of the image of any ball $\varphi_U(\textup{B}(x,r))$ must be at most $C_Ar$. Similarly $V$ is an orthogonal matrix which can be thought of a (possibly trivial) rotation of the plane spanned by $(1,0,0)$ and $(0,1,0)$ followed by a rotation which sends the direction spanned by $(1,0,0)$ to the direction $\mathbf{v}_A$, thus there exists a constant $C_A'=O(|v_A|)$ such that the diameter of the image of any ball $\varphi_V(\textup{B}(x,r))$ must be at most $C_A'r$. Since $u_A$ and $v_A$ belong to the compact set $F^{-1}(\mathrm{P}(K))$, we are done.\end{proof}

Now, fix $s > s_\A$. Note it is sufficient to assume $s_\A \in [0,2)$. We will assume that $s_\A \in [1,2)$ and comment on how to make the necessary adaptations for the other case at the end. Fix $\delta>0$ and let $B \subset \R^2$ be any ball which contains $K_\A$. Let $r(B)$ denote the radius of the ball $B$, noting that $r(B) =O( \mathrm{diam} K_\A)$. Consider
$$\I_\delta=\left\{\i \in \I^*: \frac{\alpha_3(A_\i)}{\alpha_1(A_\i)} \leq \delta< \inf_{1 \leq n \leq |\i|-1}\frac{\alpha_3(A_{\i|n})}{\alpha_1(A_{\i|n})}\right\}.$$
Let $\i \in \I_\delta$ and let $A_\i=V_\i D_\i U_\i$ be the singular value decomposition of $A_\i$ as in \eqref{svd}.

Since $\A$ is a finite subset of $\SL_{>0}$, Lemma \ref{singvector} is applicable, which infers that $\varphi_{U_\i}(B)$ is contained in a ball $B'_\i$ of radius $r(B'_\i) \leq Cr(B)$ where $0<C(K)=C<\infty$ is a constant which depends only on $K$. Since $\varphi_{D_\i}(x,y)=(\frac{\alpha_2(A_\i)}{\alpha_1(A_\i)} x, \frac{\alpha_3(A_\i)}{\alpha_1(A_\i)}y)$, $\varphi_{D_\i}(B_\i')$ is an ellipse with semiaxes of length $\frac{\alpha_2(A_\i)}{\alpha_1(A_\i)}r(B_\i')$ and $\frac{\alpha_3(A_\i)}{\alpha_1(A_\i)}r(B'_\i)$. So $\varphi_{D_\i}(B'_\i)$ can be covered by a collection $B_{\delta,\i}$ of balls of radius $\frac{\alpha_3(A_\i)}{\alpha_1(A_\i)}r(B'_\i)$, where $\# B_{\delta,\i}=O( \frac{\alpha_2(A_\i)}{\alpha_3(A_\i)})$. For each ball $B'' \in B_{\delta,\i}$, $\varphi_{V_\i}(B'')$ is contained in a ball of radius $C\frac{\alpha_3(A_\i)}{\alpha_1(A_\i)}r(B'_\i)$ by Lemma \ref{singvector}. 

Since $K_\A= \bigcup_{\i \in\I_\delta} \varphi_{A_\i}(K_\A)$, $\varphi_{A_\i}=\varphi_{V_\i}\circ\varphi_{D_\i}\circ\varphi_{U_\i}$ and by definition of $\I_\delta$,
$$\bigcup_{\i \in \I_\delta} \bigcup_{B'' \in B_{\delta,\i}} \varphi_{V_\i}(B'')$$
is a cover of $K_\A$ by sets of diameter at most $C^2 \delta r(B)$. Hence the approximate Hausdorff measure
$$\mathcal{H}^s_{C^2\delta r(B)}(K_\A) \leq \sum_{\i \in \I_\delta}C^{2s}\frac{\alpha_2(A_\i)}{\alpha_1(A_\i)}\left(\frac{\alpha_3(A_\i)}{\alpha_1(A_\i)}\right)^{s-1}r(B)^s \leq C^{2s}r(B)^s \zeta_\A(s)$$
hence $\mathcal{H}^s(K_\A) <\infty$ and $\hd K_\A \leq s$. This proves the result in the case that $s_\A \in [1,2)$.

In the case that $s_\A \in [0,1)$ we modify the argument by replacing $\I_\delta$ by 
$$\I_\delta'=\left\{\i \in \I^*: \frac{\alpha_2(A_\i)}{\alpha_1(A_\i)} \leq \delta< \inf_{1 \leq n \leq |\i|-1}\frac{\alpha_2(A_{\i|n})}{\alpha_1(A_{\i|n})}\right\}$$ and we instead cover each $\varphi_{D_\i}(B'_\i)$ by a single ball $B''_\i$ of radius $\frac{\alpha_3(A_\i)}{\alpha_1(A_\i)} r(B'_\i)$. Then following the argument as before we obtain 
$$\mathcal{H}^s_{C^2\delta r(B)}(K_\A) \leq \sum_{\i \in \I_\delta'}C^{2s}\left(\frac{\alpha_2(A_\i)}{\alpha_1(A_\i)}\right)^sr(B)^s \leq C^{2s}r(B)^s \zeta_\A(s)$$
again giving that $\hd K_\A \leq s$.

\begin{rem}
Note that the only time the finiteness of $\A$ is used is to deduce that the hypothesis of Lemma \ref{singvector} is satisfied. This means that the upper bound also holds in the case that $\A \subset \SL_{>0}$ is countable and satisfies the uniform cone contraction hypothesis of Lemma \ref{singvector}.
\end{rem}

\subsection{Proof of Theorem \ref{rauzythm}}

In this section we will restrict to the set $\A_R$ \eqref{rauzysystem} which generates the Rauzy gasket $R=K_{\A_R}$. Since $\A_R$ is not a subset of $\SL_{>0}$, Theorem \ref{setthm} will not be applicable directly. However we will show that $\hd R$ can be approximated by a sequence $\hd K_{\Gamma_N}$ where each $\Gamma_N\subset \Se$ is a finite set which can be simultaneously conjugated to a subset of $\SL_{>0}$, which is therefore amenable to Theorem \ref{setthm}.

Denote
$$\Gamma:= \{A_1^nA_2, A_1^nA_3, A_2^nA_1, A_2^nA_3, A_3^nA_1, A_3^nA_2\}_{n \geq 1} \subset \Se.$$

Let $\Gamma_N$ be an increasing sequence of finite subsets of $\Gamma$ with the property that $\Gamma=\bigcup_N \Gamma_N$. 

\begin{prop} \label{gammadom}
\begin{enumerate}
\item   $R \setminus K_\Gamma$ is a countable set of points.
\item $\Gamma$ can be simultaneously conjugated to a set of matrices which satisfies the hypothesis of Proposition \ref{submult}.
%\item There exist closed cones $C_1, C_2\subset \R^3$ such that $C_2 \subset \mathrm{int}\; C_1$ and each $A \in \Gamma$ maps $A(C_1) \subset C_2$.
%\item There exists $r \in (0,1)$ such that for all $A \in \Gamma^n$, $\frac{\alpha_2(A)}{\alpha_1(A)} \leq r^n$.
\end{enumerate}
\end{prop}

\begin{proof}
To prove (1) we appeal to the underlying symbolic coding. Denote $\Sigma=\{1,2,3\}^\N$ and $\Sigma'=\Gamma^\N$ so that $R=\Pi(\Sigma)$ and $K_\Gamma=\Pi(\Sigma')$ where $\Sigma \setminus \Sigma'$ consists of all sequences which end in either $1^\infty$, $2^\infty$ or $3^\infty$. Since there are only countably many such sequences, the claim follows.

Next we turn to proving (2).  For $i \neq j$, $A_i^nA_j$ is characterised as follows: the $i$th row is defined by: $a_{ij}=n$, $a_{ii}=n+1$ and $a_{ik}=2n$ (where $k \neq i,j$); the $j$th row is $(1,1,1)$ and the $k$th row is $e_k$. For example, 
$$A_1^nA_2= \begin{pmatrix} n+1&n&2n \\1&1&1\\ 0&0&1 \end{pmatrix}.$$

%We will show that the matrices in $\Gamma$ can be simultaneously conjugated by some matrix to obtain a set of positive matrices each of which map the closure of the positive cone $\Delta$ into some common compact subset $K \subset \Delta^{\mathrm{o}}$. To this end, consider the matrix
We will conjugate each matrix in $\Gamma$ by
$$M_\epsilon= \begin{pmatrix} 1&-\epsilon&-\epsilon \\-\epsilon&1&-\epsilon\\ -\epsilon&-\epsilon&1 \end{pmatrix}$$
for some sufficiently small $\epsilon$ in order to verify (2). Let us write $X=\Theta(n)$ if there exist constants $0<c_1<c_2<\infty$ which depend only on $\Gamma$ such that $c_1 \leq \frac{X}{n} \leq c_2$ and $X=\Theta_\epsilon(n)$ if there exist constants $0<c_1<c_2<\infty$ which depend only on $\Gamma$ and $\epsilon$ such that $c_1 \leq \frac{X}{n} \leq c_2$.  It is sufficient to find $\epsilon>0$ such that each entry in any matrix in the family $\{M_\epsilon^{-1}AM_\epsilon\}_{A \in \Gamma}$ is equal to $\Theta_\epsilon(n)$. 

It is easy to see that for $\epsilon\leq\frac{1}{5}$, for any $A_i^nA_j \in \Gamma$, $A_i^nA_jM_\epsilon$ will have the $i$th row consisting of entries $\Theta(n)$, the $j$th row is $(1-2\epsilon,1-2\epsilon,1-2\epsilon)$ and the $k$th row ($k \neq i,j$) is $e_k$, e.g.
$$A_1^nA_2M=\begin{pmatrix}n+1-3n\epsilon& n-(3n+1)\epsilon&2n-(2n+1)\epsilon \\1-2\epsilon&1-2\epsilon&1-2\epsilon\\ 0&0&1 \end{pmatrix}.$$
Fix $\epsilon=\frac{1}{5}$. Since 
$$M_\epsilon^{-1}=\frac{1}{1-\epsilon^2}\begin{pmatrix} 1-\epsilon^2&\epsilon+\epsilon^2&\epsilon+\epsilon^2 \\ \epsilon+\epsilon^2&1-\epsilon^2&\epsilon+\epsilon^2\\ \epsilon+\epsilon^2&\epsilon+\epsilon^2&1-\epsilon^2 \end{pmatrix}$$ 
it is easy to see that for any $A \in \Gamma$, $M_\epsilon^{-1}AM_\epsilon$ has all of its entries equal to $\Theta_\epsilon(n)$.  This completes the proof of (2).
%So we have proved that $\Gamma$ has a strictly invariant multicone.

\end{proof}
\begin{cor}\label{approx}
$\sup_N s_{\Gamma_N}=s_\Gamma$.
\end{cor}

\begin{proof}
Since $\Gamma_N \subset \Gamma$, clearly $s_{\Gamma_N} \leq s_\Gamma$. Now, let $s<s_\Gamma$. By Proposition \ref{gammadom}(2) and Proposition \ref{submult}, the singular value function $\phi^s$ is almost-submultiplicative on $\Se_\Gamma$ and $\Se_N$ hence $s_\Gamma$ and $s_{\Gamma_N}$ are unique zeros of $P_\Gamma$ and $P_{\Gamma_N}$ respectively. Hence it is sufficient to show that $P_{\Gamma_N}(s)>0$ for some $N \in \N$. 

By the quasimultiplicativity (Proposition \ref{quasiprop}) and almost-submultiplicativity (Lemma \ref{submult}) of $\phi^s$, we can apply \cite[Proposition 3.2]{kaenmaki2014multifractal} to deduce that $P_\Gamma(s)=\sup_N P_{\Gamma_N}(s)$, from which the claim follows.\end{proof}

Let $\Se_R$ denote the semigroup generated by $\A_R$, $\Se_\Gamma$ denote the semigroup generated by $\Gamma$ and $\Se_N$ denote the semigroup generated by $\Gamma_N$. The following lemma determines that to prove Theorem \ref{rauzythm} it will suffice to prove that the dimension of $K_\Gamma$ is given in terms of $s_\Gamma$.

\begin{lem}\label{gamma=a}
$s_{\A_R}=s_\Gamma$.
\end{lem}

\begin{proof}
Since $\Gamma \subset \Se_R$, clearly $s_\Gamma \leq s_{\A_R}$. 

To see the other direction, note that
\begin{align*}
\zeta_{\A_R}(s)&=\zeta_\Gamma(s)+\sum_{\i \in \Gamma^n} \sum_{n \in \N} \phi^s(A_\i A_1^n)+ \phi^s(A_\i A_2^n)+ \phi^s(A_\i A_3^n)\\
&\leq C \zeta_\Gamma(s)+\sum_{\i \in \Gamma^n} \sum_{n \in \N} \phi^s(A_\i A_1^nA_2)+ \phi^s(A_\i A_2^nA_1)+ \phi^s(A_\i A_3^nA_1)\\
&\leq 2C\zeta_\Gamma(s)
\end{align*}
where in the second line we have used that there exists $C<\infty$, depending only on $A_1,A_2,A_3$ and $s$, such that for all $i \in \{1,2,3\}$ and $B \in \SL$, $\phi^s(B) \leq C\phi^s(BA_i)$, see e.g. \cite[Lemma 1]{bochi2009some}. Thus $s_{\A_R} \leq s_\Gamma$ completing the proof.

\end{proof}

Before proving the main result, the final thing left to prove is that for sufficiently large $N$, $\Se_N $ is Zariski-dense in $\SL$. 

\begin{prop}
For sufficiently large $N$ $\Se_N $ is Zariski dense in $\SL$. 
\label{zariskiprop}
\end{prop}

\begin{proof}We begin by proving that $\Se_\Gamma $ is Zariski-dense in $\SL$. Let $\mathfrak{sl}(3,\R)$ denote the special linear Lie algebra which consists of $3 \times 3$  matrices with trace zero and Lie bracket $[X,Y]=XY-YX$ and which generates the Lie group $\SL$. Let $P: \SL \to \R$ be a polynomial map such that $P(A)=0$ for all $A \in \Se_\Gamma $. As we saw in Proposition \ref{gammadom}, for $i \neq j$, $A_i^nA_j$ is characterised as follows: the $i$th row is defined by: $a_{ij}=n$, $a_{ii}=n+1$ and $a_{ik}=2n$ (where $k \neq i,j$); the $j$th row is $(1,1,1)$ and the $k$th row is $e_k$. For concereteness let us take $i=1, j=2$ so
$$A_1^nA_2= \begin{pmatrix} n+1&n&2n \\1&1&1\\ 0&0&1 \end{pmatrix}.$$
For all $n \in \N$, $P(A_1^nA_2)=0$. Let $q(x)$ be the polynomial in one real variable 
$$q(x)=P\left(\begin{pmatrix} x+1&x&2x \\1&1&1\\ 0&0&1 \end{pmatrix}\right).$$
Since $q(x)=0$ for all $x \in \N$ and $q$ is a polynomial, $q \equiv 0$. Write $G$ to be the Zariski closure of $\Se_\Gamma $, recalling that $G$ is a Lie group. Thus we have
$$B_x=\begin{pmatrix} x+1&x&2x \\1&1&1\\ 0&0&1 \end{pmatrix} \in G$$
for all $x \in \R$. Now consider the map $x \mapsto B_x$ which is a smooth algebraic curve in $G$, and in particular
$$M_1=\frac{d}{dx} B_x |_{x=0} = \begin{pmatrix} 1&1&2\\0&0&0\\ 0&0&0 \end{pmatrix} \in \mathfrak{g}$$
where $\mathfrak{g}$ denotes the Lie algebra of $G$. By following this reasoning for the other generators $A_i^nA_j$ we also deduce that
$$C_2=\begin{pmatrix} 1&2&1 \\0&0&0\\ 0&0&0\end{pmatrix}, C_3=\begin{pmatrix} 0&0&0\\2&1&1 \\ 0&0&0\end{pmatrix},C_4=\begin{pmatrix} 0&0&0\\1&1&2 \\ 0&0&0\end{pmatrix}, C_5=\begin{pmatrix} 0&0&0 \\ 0&0&0\\2&1&1\end{pmatrix}, C_6=\begin{pmatrix} 0&0&0 \\ 0&0&0\\1&2&1\end{pmatrix} \in \mathfrak{g}.$$
Moreover, by taking Lie brackets we deduce that
$$C_7=[C_1,C_5]=C_1C_5-C_5C_1=\begin{pmatrix} 4&2&2 \\ 0&0&0\\-2&-2&-4\end{pmatrix} \in \mathfrak{g}$$
and
$$C_8=[C_2,C_4]=C_2C_4-C_4C_2=\begin{pmatrix} 2&2 &4\\ -1&-2&-1\\0&0&0\end{pmatrix}\in \mathfrak{g}.$$
It is easy to see that $\{C_i\}_{i=1}^8$ is a linearly independent subset of $\mathfrak{g}$. However, since $\mathfrak{g}$ must be a subalgebra of the Lie algebra of $\SL$, which is 8 dimensional, this implies that $\mathfrak{g}$ is the Lie algebra of $\SL$ hence $\Se_\Gamma $ is Zariski dense in $\SL$. 

Finally, to prove the proposition note that $\Se_\Gamma$ is a subsemigroup of $ \Se_R$ and so the Zariski closure of $\Se_R$ is $\SL$, which is a Zariski closed and connected subset of $\GL$, thus the claim for $\Gamma_N$ with $N$ sufficiently large follows from \cite[Lemma 3.7]{morris2023variational}.
\end{proof}

We can now prove Theorem \ref{rauzythm}. For the lower bound,
$$\hd R \geq \sup_N \hd K_{\Gamma_N}=\sup_N s_{\Gamma_N}=s_\Gamma=s_{\A_R}$$
where we have used that for sufficiently large $N$, each $\Gamma_N$ is simultaneously conjugate to a subset of $\SL_{>0}$ which satisfies the SOSC and such that $\Se_N$ is Zariski dense in $\SL$, along with Theorem \ref{setthm}, as well as Corollary \ref{approx} and Lemma \ref{gamma=a}.

 For the upper bound 

$$\hd R=\hd K_\Gamma\leq s_\Gamma=s_{\A_R}$$
where the first equality is because $R\setminus K_\Gamma$ is countable and the final by Lemma \ref{gamma=a}. To see the second, we employ almost an identical covering argument to \S \ref{sec:UB}. Indeed, by Proposition \ref{gammadom}, $\Gamma$ can be simultaneously conjugated to a subset $\Gamma_\epsilon \subset\SL_{>0}$ which maps $\R^2_{>0}$ into a compact subset of itself, therefore Remark \ref{singvector} is also applicable to $\Gamma_\epsilon$ and the proof follows exactly as it did in \S \ref{sec:UB}.

\bibliographystyle{plain}
\bibliography{bibfile}

\end{document}